\documentclass[12pt]{amsart}
\usepackage[english]{babel}
\usepackage[utf8]{inputenc}
\usepackage{times}
\usepackage[T1]{fontenc}
\usepackage{graphicx}
\usepackage{amscd,amsmath,amsfonts,amstext,amssymb,amsbsy,amsopn,amsthm,eucal}
\usepackage{txfonts}

\usepackage{fancyhdr}
\usepackage{hyperref}

\hypersetup{
  pdftitle={Sharp estimate on the first eigenvalue of the p-Laplacian},
  pdfauthor={Daniele Valtorta},
  pdfsubject={Sharp estimate on the first eigenvalue of the p-Laplacian on a compact manifold with nonnegative Ricci curvature and (possibly) convex boundary and Neumann Boundary conditions},
  pdfkeywords={first eigenvalue, spectral gap, p-Laplacian, nonnegative Ricci curvature}
  pdfpagelayout=SinglePage,
  pdfpagemode=UseOutlines,  pdfpagemode=UseOutlines,
  colorlinks,
  linkcolor=[rgb]{0,0,0.7},
  urlcolor=[rgb]{0,0,0.4},
  citecolor=[rgb]{0.4,0.1,0}
  }

\newcommand{\R}{\mathbb{R}}

\newcommand{\ps}[2]{\left\langle#1\middle\vert#2\right\rangle}
\newcommand{\pps}[2]{\langle\left\langle#1\middle\vert#2\right\rangle\rangle}
\newcommand{\psp}{\ps{\cdot}{\cdot}}
\newcommand{\ppsp}{\pps{\cdot}{\cdot}}
\newcommand{\pst}[3]{#2\ton{#1,#3}}

\newcommand{\ton}[1]{\left(#1\right)}
\newcommand{\qua}[1]{\left[#1\right]}
\newcommand{\cur}[1]{\left\{#1\right\}}
\newcommand{\abs}[1]{\left|#1\right|}

\newcommand{\dive}[1]{\operatorname{div}\ton{#1}}
\newcommand{\Ric}{\operatorname{Ric}}

\newcommand{\V}{\operatorname{Vol}}
\renewcommand{\L}{\operatorname{L}}

\newcommand{\pdw}{\dot w ^{(p-1)}}
\newcommand{\pw}{w^{(p-1)}}
\newcommand{\pG}{G^{(p-1 )}}
\newcommand{\pu}{ u^{(p-1)}}
\newcommand{\pf}{Pr\"{u}fer }
\newcommand{\sinp}{\operatorname{sin_p}}
\newcommand{\cosp}{\operatorname{cos_p}}

\newtheorem{prop}{Proposition}[section]
\newtheorem{teo}[prop]{Theorem}
\newtheorem{deph}[prop]{Definition}
\newtheorem{lemma}[prop]{Lemma}

\newtheorem{rem}[prop]{Remark}
\newtheorem{cor}[prop]{Corollary}

\numberwithin{equation}{section}

\begin{document}


\baselineskip=17pt


\title[Sharp estimate on the first eigenvalue of the p-Laplacian]{Sharp estimate on the first eigenvalue of the p-Laplacian on compact manifold with nonnegative Ricci curvature}

\author[Daniele Valtorta]{Daniele Valtorta}

\date{01/02/2012, e-mail: \href{mailto:danielevaltorta@gmail.com}{danielevaltorta@gmail.com}}

\begin{abstract}
 We prove the sharp estimate on the first nonzero eigenvalue of the p-Laplacian on a compact Riemannian manifold with nonnegative Ricci curvature and possibly with convex boundary (in this case we assume Neumann b.c. on the p-Laplacian). The proof is based on a gradient comparison theorem. We will also characterize the equality case in the estimate.
\end{abstract}


\maketitle

\tableofcontents

This article is a preprint of the article published on \href{http://www.sciencedirect.com/science/article/pii/S0362546X12001411}{Nonlinear Analysis}: Theory, Methods \& Applications. Volume 75, Issue 13, September 2012, Pages 4974–4994. Although there are differences with the published version, the content of this preprint is substantially the same.

\section{Introduction}
Let $M$ be a compact Riemannian manifold with nonnegative Ricci curvature and possibly with convex boundary, and let $d$ be its diameter. For a function $u\in W^{1,p}(M)$, we define its $p$-Laplacian as
\begin{gather}
  \Delta_p(u)\equiv \operatorname{div}(\abs{\nabla u}^{p-2} \nabla u)\ ,
\end{gather}
where the equality is in the weak $W^{1,p}(M)$ sense. We will denote the first positive eigenvalue of this operator as $\lambda_p$, assuming Neumann boundary conditions on the boundary if necessary. In particular, $\lambda_p$ is the smallest positive real number such that there exists a nonzero $u\in W^{1,p}(M)$ satisfying in the weak sense
\begin{gather}\label{eq_plap}
\begin{cases}
 \Delta_p(u)= -\lambda_p \abs u ^{p-2} u \ \ \ \ \ &on \ M\\
 \ps{\nabla u}{\hat n}=0 \ \ \  \ \ \ \ &on \ \partial M \ .
\end{cases}
\end{gather}
In this article, we will prove the following sharp estimate:
\begin{gather}\label{eq_1}
 \frac{\lambda_p }{p-1}\geq \frac{\pi_p^p}{d^p}                              \ ,
\end{gather}
where $$\pi_p =\int_{-1}^1 \frac{ds}{(1-\abs s ^p)^{1/p}}=\frac{2\pi}{p\sin(\pi/p)} \ .$$
Moreover, we will also prove that equality in this estimate can occur if and only if $M$ is a one dimensional circle or a segment.

In the case where $p=2$, this problem has been intensively studied, in particular in \cite{ZY} the sharp estimate
\begin{gather*}
 \lambda_2 \geq \frac{\pi^2}{d^2}
\end{gather*}
is obtained assuming that $M$ has nonnegative Ricci curvature.

The main tool used in this article is a gradient estimate for the function $u$, technique which was used by P. Li and S. T. Yau to get eigenvalue estimates for the usual Laplacian (see \cite{LY} and also \cite{SYred}). They were able to prove that:
\begin{lemma}
 On a compact manifold $M$ with nonnegative Ricci curvature, if a function $u$ is such that $\Delta u = -\lambda_2 u$ and $\abs u \leq 1$, the following estimate is valid where $u\neq \pm1$
\begin{gather*}
 \frac{\abs{\nabla u}^2}{1-u^2}\leq \lambda_2   \ .
\end{gather*}
Note that where $u=\pm 1$, $\nabla u=0$.
\begin{proof}[{Sketch of the proof}]
 The proof is based on a very common argument: consider the function $F\equiv \frac{\abs{\nabla u}^2}{1+\epsilon-u^2}$ on the manifold $M$. Necessarily $F$ attains a maximum, and at this point $\nabla F=0$ and $\Delta F\leq 0$. From these two relations, one proves that $F\leq \lambda_2$.
\end{proof}
\end{lemma}
With this gradient estimate Li and Yau proved that
\begin{gather*}
 \lambda_2 \geq \frac{\pi^2}{4d^2} \ .
\end{gather*}
For the reader's convenience, we briefly sketch the proof of this estimate. Rescale the eigenfunction $u$ in such a way that $m=\min\{u\}=-1$ and $0<M=\max\{u\}\leq 1$ and consider a unit speed minimizing geodesic $\gamma$ joining a minimum point $x_-$ and a maximum point $x_+$ for $u$, then a simple change of variables yields:
\begin{gather*}
\frac{\pi}{2}=\int_{-1}^0\frac{du}{\sqrt{1-u^2}} < \int_{m}^M \frac{du}{\sqrt{1-u^2}}  \leq \int_{\gamma}\frac{\abs{\nabla u}}{\sqrt{1-u^2}} dt \int_{\gamma} \frac{\abs{\nabla u}}{\sqrt{1-u^2}} dt \leq \sqrt{\lambda_2} d
\end{gather*}
Note that the strict inequality in this chain forces this estimate to be non-sharp, inequality which arises from the fact that $\max\{u\}=M>0$. If in addition we suppose that $M=1$, we can improve this estimate and get directly the sharp one. This suggests that it is important to consider the behaviour of the maximum of the eigenfunction to improve this partial result. In fact Li and Yau were able to sharpen their estimate by using the function $F\equiv \frac{\abs{\nabla u}^2}{(u-m)(M-u)}$ for their gradient estimate, which led them to prove that $\lambda_2\geq \frac{\pi^2}{2d^2}$.

J. Zhong and H. Yang obtained the sharp estimate using a barrier argument to improve further the gradient estimate (see \cite{ZY} and also \cite{SYred}).

Later on M. Chen and F. Wang in \cite{chen} and \cite{chen-} and independently P. Kroger in \cite{kro} (see also \cite{kro2} for explicit bounds) with different techniques were able to estimate the first eigenvalue of the Laplacian by using a one dimensional model. Note that their work also applies to generic lower bounds for the Ricci curvature. The main tool in \cite{chen-} is a variational formula, while \cite{kro} uses a gradient comparison technique. This second technique was also adapted by D. Bakry and Z. Qian in \cite{new} to obtain eigenvalue estimates for weighted Laplace operators with assumptions on the Bakry-Emery Ricci curvature. In this article we will follow this latter technique based on the gradient comparison. Roughly speaking, the basic idea is to find the right function $w:\R\to \R$ such that $\abs{\nabla u}\leq \abs{\dot w}|_{w^{-1}(u)}$ on $M$. In order to find what conditions $w$ must satisfy and to prove the gradient comparison, Bakry and Qian use, among other instruments, some 
estimates related to the Bochner formula. For the sake of extending this estimates in our setting, we will prove a generalized version valid for any $p\in (1,\infty)$ of this well-known formula.

In the generic $p$ case, some estimates on the first eigenvalue of the $p$-Laplacian are known, in particular see \cite{hui} and \cite{kn}; \cite{take} presents different kind of estimates, and for a general review on the problem with a variational twist see \cite{le}. In \cite{hui} and \cite{kn} the general idea of the estimate is the same as in the linear case, in fact the authors get a gradient estimate via the maximum principle, but instead of using the usual Laplacian in $\Delta F\leq 0$ at the maximum point, they use the linearized $p$-Laplacian, which will be introduced later in this work.

By estimating the function $F=\frac{\abs{\nabla u}^2}{1-u^2}$, \cite{kn} is able to prove that on a compact Riemannian manifold with $\Ric \geq 0$ and for $p\geq 2$
\begin{gather*}
 \lambda_p\geq \frac 1 {p-1} \ton{\frac \pi {4d} }^p   \ ,
\end{gather*}
while \cite{hui} uses $F=\frac{\abs{\nabla u}^p}{1-u^p}$ and assumes that the Ricci curvature is quasi-positive (i.e. $\Ric\geq 0$ on $M$ but with at least one point where $\Ric >0$), to prove that for $p>1$
\begin{gather*}
 \lambda_p \geq (p-1) \ton{\frac{\pi_p}{2d}}^p \ .
\end{gather*}

The estimate proved in this article is better than both these estimates and it is sharp. In fact, as we will see, on any one dimensional circle or segment the first nontrivial eigenvalue of the $p$-Laplacian is exactly $\frac{\lambda_p }{p-1}= \ton{\frac{\pi_p}{d}}^p$.


As for the applications of this result, recall that the first eigenvalue of the $p$-Laplacian is related to the Poincar\'e constant, which is by definition
\begin{gather*}
 C_p=\inf\left\{\frac{\int_M \abs{\nabla u}^p d\V}{\int_M \abs u^p d\V} \ \text{ with }u\in M \ s.t. \ \ \int_M \abs u ^{p-2} u \ d\V=0\right\} \ .
\end{gather*}
In particular by standard variational techniques one shows that $C_p=\lambda_p$, so a sharp estimate on the first eigenvalue is of course a sharp estimate on the Poincar\'e constant. Recall also that in the case of a manifold with boundary this equivalence holds if one assumes Neumann boundary conditions on the $p$-Laplacian. Using different techniques and in Euclidean setting, a sharp estimate similar to the one in this article has been obtained independently in \cite{carlo}.

Other applications (surprisingly also of practical interest) related to the $p$-Laplacian are discussed in \cite[Pag. 2]{W} and \cite{Diaz}.

It is worth mentioning that very recent studies have been made on the connection of the first eigenvalue of the $p$-Laplacian with the Ricci flow, see \cite{flow}.

The article is organized as follows: first we briefly discuss the case $n=1$ where the eigenfunction assumes an explicit form, then we define the linearized $p$-Laplacian and prove a sort of $p$-Bochner formula. Using some technical lemmas needed to study the one dimensional model functions, we will be able to state and prove the gradient comparison theorem, and as a consequence also the main theorem on the spectral gap, which is:
\begin{teo}\label{teo_1}
 Let $M$ be a compact Riemannian manifold with nonnegative Ricci curvature, diameter $d$ and possibly with convex boundary. Let $\lambda_p$ be the first nontrivial (=nonzero) eigenvalue of the $p$-Laplacian (with Neumann boundary condition if necessary), i.e.
\begin{gather*}
 \Delta_p (u) = -\lambda_p \abs u^{p-2} u
\end{gather*}
for some nonconstant function $u$. Then the following sharp estimate holds
\begin{gather*}
 \frac{\lambda_p}{p-1} \geq \frac{\pi_p^p}{d^p} \ .
\end{gather*}
Moreover a necessary (but not sufficient) condition for equality to hold in this estimate is that \[\max\{u\}=-\min\{u\}\, .\]
\end{teo}

\noindent The characterization of the equality case is dealt with in the last section. In \cite{Hang}, this characterization is proved in the case where $p=2$ to answer a problem raised by T. Sakai in \cite{saka}. Unfortunately, this proof relies on the properties of the Hessian of a 2-eigenfunction, which are not easily generalized for generic $p$.

\subsection{Positive and negative lower bounds on Ricci}
If the manifold $M$ has Ricci curvature bounded from below by a positive constant, the sharp estimate for $\lambda_{1,p}$ is obtained in \cite{matei}, where the author uses Levy-Gromov isoperimetric inequality to prove a generalized version of Obata's theorem. For negative lower bounds on Ricci, the sharp estimate (for generic $p\in (1,\infty)$) is proved in the later work \cite{nava}.

\subsection{Notation}\label{sec_not}
We will use the following conventions. $(M,\ps \cdot \cdot)$ will indicate a Riemannian manifold with nonnegative Ricci curvature, diameter $d$ and dimension $n$. Throughout the article, we fix $p>1$ (so we will write $\lambda$ for $\lambda_p$), and we will define for any $w\in \R$
\begin{gather*}
 w^{(p-1)}\equiv \abs{w}^{p-2}w= \abs{w}^{p-1} \operatorname{sign}(w) \ .
\end{gather*}
Given a function $u:M\to \R$, $H_u$ will denote its Hessian where defined, and we set
\begin{gather*}
 A_u\equiv \frac{\pst{\nabla u}{H_u }{\nabla u}}{\abs{\nabla u}^2} \ .
\end{gather*}
We will use the convention
\begin{gather*}
 u_{ij}= \nabla_j \nabla_i u 
\end{gather*}
and the Einstein summation convention. We will consider the Hessian as a $(2,0)$ or $(1,1)$ tensor, so for example
\begin{gather*}
 H_u(\nabla u,\nabla f)= u_{ij}u^i f^j \ .
\end{gather*}
$\abs{H_u}$ will indicate the Hilbert-Schmidt norm of $H_u$, so that
\begin{gather*}
 \abs{H_u}^2= u_{ij}u^{ij} \ .
\end{gather*}

In the following we will (sometimes implicitly) use the regularity theorems valid for solutions of equation \eqref{eq_plap}. 
In general, the solution belongs to $W^{1,p}(M)\cap C^{1,\alpha}(M)$ for some $\alpha>0$, and elliptic regularity ensures that $u$ is a smooth function where $\nabla u \neq 0$ and $u\neq 0$. If $\nabla u(x)\neq 0$ and $u(x)=0$, then $u\in C^{3,\alpha}(U)$ if $p>2$ and $u\in C^{2,\alpha}(U)$ for $1<p<2$, where $U$ is a suitably small neighborhood of $x$. The standard reference for these results is \cite{reg}, where the problem is studied in local coordinates.

\section{One dimensional p-Laplacian}\label{sec_1d}
The first nontrivial eigenfunction of the p-Laplacian is very easily found if $n=1$. In this case it is well-known that $M$ is either a circle or a segment, moreover equation \eqref{eq_plap} assumes the form
\begin{gather}\label{eq_plap1}
 (p-1)\abs{\dot u}^{p-2} \ddot u +\lambda u^{(p-1)}=0 \ .
\end{gather}
In order to study this eigenvalue problem, we define the function $\sinp(x)$ on $\qua{-\frac {\pi_p}2,\frac{3\pi_p}2}$ by
\begin{gather*}
 \begin{cases}
x=\int_0^{\sin_p(x)} \frac{ds}{(1-s^p)^{1/p}} & \text{  if } x\in \qua{-\frac{\pi_p}{2},\frac{\pi_p}{2}}\\
\sin_p(\pi_p-x) & \text{  if } x\in \qua{\frac{\pi_p}{2},\frac{3 \pi_p}{2}}
 \end{cases}
\end{gather*}
and extend it on the whole real line as a periodic function of period $2\pi_p$. It is easy to check that for $p\neq 2$ this function is smooth around noncritical points, but only $C^{1,\alpha}(\R)$ for $\alpha=\min\{p-1,(p-1)^{-1}\}$. For a more detailed study of the $p$-sine, we refer the reader to \cite{dosly} and \cite[pag. 388]{pino2}.

Define the quantity
\begin{gather}
e(x)= \ton{\abs{\dot u }^{p} + \frac{\lambda \abs u^p}{p-1}}^{\frac 1 p} .
\end{gather}
If $u$ is a solution to \eqref{eq_plap1}, then $e$ is constant on the whole manifold, so by integration we see that all solutions of \eqref{eq_plap1} are of the form $A\sinp(\lambda^{1/p} x+B)$ for some real constants $A,B$. Due to this observation our one-dimensional eigenvalue problem is easily solved.

In fact, identify the circumference of length $2d$ with the real interval $[0,2d]$ with identified end-points. It is easily seen that the first eigenfunction on this manifold is, up to translations and dilatations, $u= \sin_p (\alpha x)$, where $\alpha= \frac {\pi_p}{d}$. Then by direct calculation we have
\begin{gather}
 \frac{\lambda}{p-1}= \ton {\frac {\pi_p}{d}}^p \ .
\end{gather}

The case with boundary (i.e. the one-dimensional segment) is completely analogous, so at least in the $n=1$ case the proof of Theorem \ref{teo_1} is quite straightforward.

\begin{rem}
Note that in this easy case, the absolute values of the maximum and minimum of the eigenfunction always coincide and the distance between a maximum and a minimum is always $d=\frac{\pi_p}{\alpha}$. Note also that if we call $\cosp(x)\equiv \frac d {dx} \sinp(x)$, then the well known identity $\sin^2(x)+\cos^2(x)=1$ generalizes to $\abs{\sinp(x)}^p+\abs{\cosp(x)}^p=1$. 
\end{rem}

\section{Linearized p-Laplacian and p-Bochner formula}
In this section we introduce the linearized operator of the $p$-Laplacian and study some of its properties.

First of all, we calculate the linearization of the p-Laplacian near a function $u$ in a naif way, i.e., we define
\begin{gather*}
 P_u(\eta)\equiv \left.\frac{d}{dt}\right\vert_{t=0} \Delta_p(u+t\eta)=\\
=\dive{(p-2) \abs{\nabla u}^{p-4}\ps{\nabla u}{\nabla \eta}\nabla u + \abs{\nabla u }^{p-2}\nabla \eta}=\\
=(p-2)\Delta_p(u)\frac{\ps{\nabla u}{\nabla \eta}}{\abs{\nabla u }^2}+(p-2)\abs{\nabla u}^{p-2}\ps{\nabla u }{\nabla \frac{\ps{\nabla u}{\nabla \eta}}{\abs{\nabla u }^2}}+\\
+(p-2)\abs{\nabla u }^{p-4}\pst{\nabla u }{H_u}{\nabla \eta} + \abs{\nabla u}^{p-2}\Delta \eta=\\
= \abs{\nabla u}^{p-2}\Delta \eta +(p-2)\abs{\nabla u}^{p-4}\pst{\nabla u}{H_\eta}{\nabla u}+ (p-2)\Delta_p(u)\frac{\ps{\nabla u}{\nabla \eta}}{\abs{\nabla u }^2}+ \\
+ 2(p-2)\abs{\nabla u}^{p-4}\pst{\nabla u}{H_u}{\nabla \eta - \frac{\nabla u}{\abs{\nabla u}}\ps{\frac{\nabla u}{\abs{\nabla u}}}{\nabla \eta}} \ .
\end{gather*}
If $u$ is an eigenfunction of the $p$-Laplacian, this operator is defined pointwise only where the gradient of $u$ is non zero (and so $u$ is locally smooth) and it is easily proved that at these points it is strictly elliptic. For convenience, denote by $P^{II}_u$ the second order part of $P_u$, which is
\begin{gather*}
 {P_u}^{II}(\eta) \equiv \abs{\nabla u}^{p-2}\Delta \eta+(p-2)\abs{\nabla u}^{p-4}\pst{\nabla u}{H_\eta}{\nabla u} \ ,
\end{gather*}
or equivalently
\begin{gather}\label{deph_pu}
 {P_u}^{II}(\eta) \equiv \qua{\abs{\nabla u}^{p-2}\delta_i^{j} +(p-2)\abs{\nabla u}^{p-4}\nabla_i u \nabla^j u }\nabla^{i}\nabla_j \eta\ . 
\end{gather}

Note that $P_u(u)=(p-1)\Delta_p(u)$ and $P^{II} _u(u)=\Delta_p(u)$.

The main property enjoyed by the linearized $p$-Laplacian is the following version of the celebrated Bochner formula.
\begin{prop}[p-\textsc{Bochner formula}]
 Given $x\in M$, a domain $U$ containing $x$, and a function $u\in C^3(U)$, if $\nabla u |_x\neq 0$ on $U$ we have
\begin{gather*}
 \frac{1}{p} P^{II}_u(\abs{\nabla u}^{p})=\\
 \abs{\nabla u}^{2(p-2)}\{\abs{\nabla u}^{2-p}[\ps{\nabla \Delta_p u}{\nabla u}-(p-2)A_u\Delta_p u]+\\
+\abs{H_u}^2+p(p-2)A_u^2+ \operatorname{Ric}(\nabla u,\nabla u) \} \ .
\end{gather*}
In particular this equality holds if $u$ is an eigenfunction of the $p$-Laplacian, $p\geq 2$ and $\nabla u |_x \neq 0$; or also if $1<p<2$ and $\nabla u|_x \neq 0$ and $u(x)\neq 0$.
\end{prop}
\begin{proof}
Just as in the usual Bochner formula, the main ingredients for this formula are the commutation rule for third derivatives and some computations.

First, compute $\Delta (\abs{\nabla u}^p)$, and to make the calculation easier consider a normal coordinate system centered at the point under consideration. Using the notation introduced in Section \ref{sec_not} we have
\begin{gather*}
\frac 1 p \Delta(\abs{\nabla u}^p) = \nabla^i\ton{ \abs{\nabla u}^{p-2} u_{ji} u^j}=\\
=\abs{\nabla u}^{p-2} \ton{ \frac {p-2}{\abs {\nabla u}^2}u^{is}u_s u_{ik}u^k+ u_{kii} u^k + u_{ik}u^{ik}} \ .
\end{gather*}
The commutation rule now allows us to interchange the indexes in the third derivatives. In particular remember that in a normal coordinate system we have
\begin{gather}\label{ref_a}
 u_{ij}=u_{ji}\ \ \ \ \ u_{ijk}-u_{ikj}=-R_{lijk}u^l\\
\notag u_{kii}=u_{iki}=u_{iik}+ \operatorname{Ric}_{ik}u^i \ .
\end{gather}
This shows that
\begin{gather}\label{ref_b}
\frac 1 p \Delta(\abs{\nabla u}^p) = \\
\notag =\abs{\nabla u}^{p-2} \ton{ \frac {p-2}{\abs {\nabla u}^2}\abs{H_u(\nabla u)}^2+ \ps{\nabla \Delta u}{\nabla u}+ \Ric(\nabla u, \nabla u)+\abs{H_u}^2} \ .
\end{gather}
In a similar fashion we have
\begin{gather*}
\frac 1 p \nabla_i \nabla_j \abs{\nabla u}^p= (p-2)\abs{\nabla u}^{p-4} u_{is}u^s u_{jk}u^k+ \abs{\nabla u}^{p-2}(u_{kij}u^k+ u_{ik}u_j^k) \ ,
\end{gather*}
which leads us to
\begin{gather*}
 \frac 1 p \frac {\nabla_i\nabla_j (\abs{\nabla u}^p) \nabla^i u \nabla^j u }{\abs {\nabla u}^2}=\\
= \abs{\nabla u}^{p-2}\ton{(p-2) A_u^2+ \frac{\nabla_i \nabla_j \nabla_k u \ \nabla^i u \nabla ^j u \nabla^k u}{\abs{\nabla u}^2}+ \frac{\abs{H_u(\nabla u)}^2}{\abs {\nabla u}^2}} \ . 
\end{gather*}
The last computation needed is
\begin{gather*}
 \ps{\nabla \Delta_p u}{\nabla u}= \nabla_i\qua{\abs{\nabla u}^{p-2}\ton{\Delta u + (p-2) A_u}} \nabla^i u=\\
=\ps{\nabla(\abs{\nabla u}^{p-2})}{\nabla u}\abs{\nabla u}^{2-p}\Delta_p u +\\ +\abs{\nabla u}^{p-2}\qua{\ps{\nabla \Delta u}{\nabla u}+(p-2) \ps{\nabla \abs{\nabla u}^{-2}}{\nabla u} \pst{\nabla u}{H_u}{\nabla u}}+\\
+(p-2) \abs{\nabla u}^{p-4}\qua{(\nabla H_u)(\nabla u,\nabla u,\nabla u)+2 \abs{H_u(\nabla u)}^2}=(p-2)A_u\Delta_p u+ \\
\abs{\nabla u}^{p-2}\cur{\ps{\nabla \Delta u}{\nabla u} +(p-2) \qua{-2A_u^2+ \frac{\nabla_i \nabla_j \nabla_k u \nabla^i u \nabla ^j u \nabla^k u}{\abs{\nabla u}^2}+ 2 \frac{\abs{H(\nabla u)}^2}{\abs{\nabla u}^2}}} \ .
\end{gather*}
Using the definition of $P^{II}_u$ given in \eqref{deph_pu}, the p-Bochner formula follows form a simple exercise of algebra.
\end{proof}

In the proof of the gradient comparison, we will need to estimate $P^{II}_u(\abs{\nabla u}^p)$ from below.
If $\Ric\geq 0$ and $\Delta_p u =-\lambda u^{(p-1)}$, one could use the very rough estimate $\abs{H_u}^2\geq A_u^2$ to obtain
\begin{gather*}
\frac 1 p P^{II}_u(\abs{\nabla u}^{p}) \geq \\
\notag\geq (p-1)^2 \abs{\nabla u}^{2p-4}A_u^2+\lambda(p-2)\abs{\nabla u}^{p-2}u^{(p-1)}A_u - \lambda (p-1)\abs{\nabla u}^{p}\abs u ^{p-2} \ .  
\end{gather*}
This estimate is used implicitly in proof of \cite[Li and Yau, Theorem 1 p.110]{SYred} (where only the usual Laplacian is studied), and also in \cite{kn} and \cite{hui}.\\
A more refined estimate on $\abs{H_u}^2$ which works in the linear case is the following
\begin{gather*}
 \abs{H_u}^2\geq \frac{(\Delta u)^2} n + \frac{n} {n-1} \ton{\frac{\Delta u}{n} - A_u}^2 \ .
\end{gather*}
This estimate is the analogue of the curvature-dimension inequality and plays a key role in \cite{new} to prove the comparison with the one dimensional model. Note also that this estimate is the only point where the dimension of the manifold $n$ and the assumption on the Ricci curvature play their role. A very encouraging observation about the $p$-Bochner formula we just obtained is that the term $\abs{H_u}^2+p(p-2)A_u^2$ seems to be the right one to generalize this last estimate, in fact we can prove
\begin{lemma}\label{lemma_n}
At a point where $u$ is $C^2$ and $\nabla u \neq 0$ we have
\begin{gather*}
 \abs{\nabla u}^{2p-4}\ton{\abs{H_u}^2 + p(p-2)A_u^2}\geq\\
\geq \frac{(\Delta_p u)^2}{n} + \frac{n}{n-1}\ton{\frac{\Delta_p u}{n } -(p-1)\abs{\nabla u}^{p-2}A_u}^2 \ .
\end{gather*}
\end{lemma}
\begin{proof}
The proof consists only in some calculations that for simplicity can be carried out in a normal coordinate system for which $\abs{\nabla u}|_x=u_1(x)$. At $x$ we can write
\begin{gather*}
 \abs{\nabla u}^{2-p} \Delta_p(u)= \Delta u + (p-2) \frac{\pst{\nabla u}{H_u}{\nabla u}}{\abs{\nabla u}^2}=(p-1)u_{11}+ \sum_{j=2}^n u_{jj} \ .
\end{gather*}
By the standard inequality $\sum_{k=1}^{n-1} a_k^2 \geq \frac 1 {n-1} \ton{\sum_{k=1}^{n-1}a_k}^2$ we get
\begin{gather}\label{eq_Hn}
\notag \abs{H_u}^2+p(p-2)A_u^2 = (p-1)^2 u_{11}^2 + 2\sum_{j=1}^n u_{1j}^2 + \sum_{i,j=2}^n u_{ij}^2\geq\\
\geq (p-1)^2 u_{11}^2 + \frac 1 {n-1} \ton{\sum_{i=2}^n u_{ii}}^2 \ .
\end{gather}
On the other hand it is easily seen that
\begin{gather*}
 \frac{\ton{\abs{\nabla u}^{2-p} \Delta_p (u)}^2}{n} + \frac n {n-1} \ton{\frac{\abs{\nabla u}^{2-p}\Delta_p(u)}{n}- (p-1) A_u}^2= \\
= \frac 1 n \ton{(p-1) u_{11}+\sum_{i=2}^n u_{ii} }^2+ \frac n {n-1} \ton{-\frac{n-1}{n} (p-1) u_{11} + \frac 1 n \sum_{i=2}^n u_{ii}}^2=\\
= (p-1)^2 u_{11}^2 + \frac 1 {n-1} \ton{\sum_{i=2}^n u_{ii}}^2 \ .
\end{gather*}
This completes the proof.
\end{proof}

\begin{cor}\label{cor_est_pu}
 If $u$ is an eigenfunction relative to the eigenvalue $\lambda$, and at a point where $\nabla u \neq 0$ and $u\neq 0$ we can estimate
\begin{gather*}
 \frac 1 p P^{II}_u \abs{\nabla u}^p \geq \frac{\lambda^2 u^{2p-2}}{n-1} + \frac{2(p-1)\lambda}{n-1} u^{(p-1)} \abs{\nabla u}^{p-2} A_u + \frac{n}{n-1}(p-1)^2 \abs{\nabla u}^{2p-4} A_u^2+\\
-\lambda(p-1)u^{p-2}\abs{\nabla u}^p+\lambda(p-2)\abs{\nabla u}^{p-2} A_u u^{(p-1)}  \ .
\end{gather*}
The assumption $u\neq 0$ is not necessary if $p\geq 2$.
\end{cor}

Note that if we substitute $n$ with $n\leq m \in \R$ in the conclusion of Lemma \ref{lemma_n}, a simple algebraic computation shows that the conclusion still holds, in particular we have the following corollary.
\begin{cor}\label{cor_m}
 Under the hypothesis of Lemma \ref{lemma_n}, if $u$ is defined on a $n$-dimensional Riemannian manifold we have for any $n\leq m \in \R$
\begin{gather*}
 \abs{\nabla u}^{2p-4}\ton{\abs{H_u}^2 + p(p-2)A_u^2}\geq\\
\geq \frac{(\Delta_p u)^2}{m} + \frac{m}{m-1}\ton{\frac{\Delta_p u}{m } -(p-1)\abs{\nabla u}^{p-2}A_u}^2 \ .
\end{gather*}
\end{cor}
\begin{proof}
The proof follows directly from the fact that for every $x,y\in \R$
\begin{gather*}
\frac{x^2}{n} + \frac {n}{n-1}\ton{\frac x n - y}^2 -\ton{\frac{x^2}{m} + \frac {m}{m-1}\ton{\frac x m - y}^2 }=\\
=\ton{\frac 1 {n-1} - \frac 1 {m-1}}(x-y)^2\geq 0 \ .
\end{gather*}
\end{proof}

We close this section with the following computation that will be useful in the proof of the main theorem
\begin{lemma}\label{lemma_pu}
Let $x\in M$ and $U$ be a domain containing $x$. If $\phi:\R\to \R$ is a function of class $C^2$ is a neighborhood of $u(x)$ and $\nabla u|_x\neq 0$, then at $x$ we have
\begin{gather*}
 P^{II} _u (\phi(u))=\dot \phi (u) \Delta_p u + (p-1)\ddot \phi (u) \abs{\nabla u}^p \ .
\end{gather*}

\end{lemma}

\section{Gradient comparison}
In this section we prove a gradient comparison theorem that will be the essential tool to prove our main theorem. Since the proof is almost the same, we prove the theorem on a manifold without boundary, and then state the version for convex boundary and Neumann boundary conditions pointing out where the proof is different. To complete the proof, we will need some technical lemmas which, for the sake of clarity, will be postponed to the next section.
\begin{teo}[\textsc{Gradient comparison theorem}]\label{grad_est}
 Let $M$ be an $n$-dimensional compact Riemannian manifold without boundary, $u$ be a eigenfunction of the $p$-Laplacian belonging to the eigenvalue $\lambda$, and let $w$ be a solution on $(0,\infty)$ of the one dimensional ODE
\begin{gather}\label{eq_ode1}
\begin{cases}
 \frac d {dt} \dot w ^{(p-1)} - T \dot w ^{(p-1)} +\lambda w^{(p-1)}=0\\
 w(a)=-1 \quad \quad \dot w (a)=0
\end{cases}
\end{gather}
where $T$ can be either $-\frac{n-1} x$ or $T=0$, and $a\geq0$. Let $b(a)>a$ be the first point such that $\dot w (b)=0$ (so that $\dot w >0$ on $(a,b)$). If $[\min(u),\max(u)]\subset [-1,w(b)=\max(w)]$, then for all $x\in M$
\begin{gather*}
\abs{\nabla u (x)}\leq \dot w|_{w^{-1}(u(x))} \ .
\end{gather*}
\end{teo}

\begin{rem}
 \rm{The differential equation \eqref{eq_ode1} and its solutions will be studied in the following section, in particular we will prove existence and continuous dependence on the parameters for any $a\geq 0$ and the oscillatory behaviour of the solutions. Moreover, the solution always belongs to the class $C^1(0,\infty)$. 

For the sake of simplicity, we will use the following notational convention. For finite values of $a$, $w$ will be the solution to the ODE \eqref{eq_ode1} with $T=-\frac{n-1} x$, while $a=\infty$ will indicate the solution of the same ODE with $T=0$ and any $a$ as initial condition. Recall that in this latter case all the solutions are invariant under translations, so the conclusions of the theorem do not change if the starting point of the solution $w$ is changed.}
\end{rem}


\begin{proof}
First of all, in order to avoid problems at the boundary of $[a,b]$, we assume that \[[\min\{u\},\max\{u\}]\subset (-1,w(b))\, ,\] so that we only have to study our 1d-model on compact subintervals of $(a,b)$. We can obtain this by multiplying $u$ by a constant $\xi<1$. If we let $\xi\to 1$, then the original statement is proved.

Consider the function defined on the manifold $M$
\begin{gather*}
 F\equiv \psi(u)\qua{\abs{\nabla u}^p -\phi(u)} \ ,
\end{gather*}
where $\psi:\R\to \R$ is a generic positive $C^2$ functions on $M$ which will be specified later, and $\phi(u(x))=\dot w^{p} |_{u(x)}$. We want to prove that $F\leq 0$ on all of $M$.

Note that we introduced the function $\psi$ in the definition of $F$ since it is not easy to prove that $\abs{\nabla u}^p-\phi(u)\leq 0$ directly.


Let $x_m$ be a point of maximum for $F$ on $M$. If $\nabla u|_{x_m}=0$, there is nothing to prove. If also $u(x_m)\neq 0$, then $u$ is smooth around $x_m$, but if $u(x_m)=0$, then $u$ has only $C^{2,\alpha}$ regularity around $x_m$ for $1<p<2$, and $C^{3,\alpha}$ if $p\geq 2$.

In the following, we will assume that $u$ is a $C^3$ function around $x_m$, and we will explain in Remark \ref{rem_reg} how to modify the proof if this is not the case (in particular, if $1<p<2$ and $u(x_m)=0$).

Since $P_u$ is an elliptic operator at $x_m$ we have
\begin{gather*}
 \nabla F|_{x_m}=0 \ \ \ \ P_u^{II} F|_{x_m}\leq 0 \ .
\end{gather*}
The first equation above implies
\begin{gather}\label{eq_nablaf}
\nabla \qua{\abs{\nabla u}^p - \phi(u)}= -\frac{F \dot \psi}{\psi^2} \nabla u \ ,\\
\notag \abs{\nabla u}^{p-2} A_u\equiv \abs{\nabla u}^{p-2} \frac{\pst{\nabla u}{H_u}{\nabla u}}{\abs{\nabla u}^2}= -\frac{1}{p} \ton{\frac{\dot\psi}{\psi^2} F - \dot \phi} \ .
\end{gather}
In order to study the second inequality, note that
\begin{gather*}
 \nabla_i \nabla_j F = \ddot \psi [\abs{\nabla u}^p-\phi(u)]\nabla_i u \nabla_j u + \dot \psi [\abs{\nabla u}^p-\phi(u)] \nabla_i\nabla_j u+\\
+ \dot \psi \nabla_j[\abs{\nabla u}^p-\phi(u)] \nabla _i u + \dot \psi \nabla_i[\abs{\nabla u}^p-\phi(u)] \nabla _j u+\\
+ \psi\qua{\nabla_i \nabla_j \abs{\nabla u}^p - \nabla_i \nabla_j \phi(u)} \ .
\end{gather*}
Using equation \eqref{eq_nablaf}, we have at $x_m$ 
\begin{gather*}
 \nabla_i \nabla_j F = \ddot \psi [\abs{\nabla u}^p-\phi(u)]\nabla_i u \nabla_j u + \dot \psi [\abs{\nabla u}^p-\phi(u)] \nabla_i\nabla_j u+\\
- 2\frac{F\dot \psi^2}{\psi^2}\nabla_j u\nabla _i u + \psi\qua{\nabla_i \nabla_j \abs{\nabla u}^p - \nabla_i \nabla_j \phi(u)} \ .
\end{gather*}
By a straightforward calculation
\begin{gather*}
 \qua{\abs{\nabla u}^{p-2}\delta^{ij} +(p-2)\abs{\nabla u}^{p-4}\nabla^i u \nabla^j u }\nabla_i u \nabla_j u=(p-1)\abs{\nabla u}^p \ ,
\end{gather*}
and applying the definition of $P^{II}_u$ (see equation \eqref{deph_pu}), we get
\begin{gather*}
 0\geq P^{II}_u (F)=-2(p-1) F \frac{\dot \psi^2}{\psi^2} \ton{\frac{F}{\psi}+ \phi(u)}- F \frac{\dot \psi}{\psi} \lambda u^{(p-1)}+\\
+ (p-1) \frac{\ddot \psi}{\psi}F \ton{\frac F {\psi} + \phi(u)} + p \psi \ton{\frac 1 p P^{II}_u \abs{\nabla u }^p}-\psi P^{II}_u (\phi) \ ,
\end{gather*}
and using Corollary \ref{cor_est_pu} and Lemma \ref{lemma_pu} we obtain the following relation valid at $x_m$
\begin{gather*}
a_1 F^2 +a_2 F +a_3\leq 0 \ ,
\end{gather*}
where
\begin{gather}\label{eq_c}
a_3=p\psi\qua{\frac{\lambda^2 u^{2p-2}}{n-1} + \frac{n+1}{n-1}\frac{p-1}{p}\lambda u^{(p-1)} \dot \phi +\right.\\
+ \left.\frac{n (p-1)^2}{p^2(n-1)} \dot \phi^2 -\lambda (p-1)\phi \abs u^{p-2}-\frac{p-1}p \phi \ddot \phi}\notag \ ,
\end{gather}
\begin{gather}\label{eq_b}
a_2= -(p-1)\frac{\dot \psi}{\psi} \lambda u^{(p-1)} \frac{n+1}{n-1}-\lambda p(p-1) \abs u^{p-2} +\\
- \frac{2n(p-1)^2}{p(n-1)}\frac{\dot \psi}{\psi} \dot \phi -(p-1) \ddot \phi +(p-1)\phi\ton{\frac{\ddot\psi}{\psi}-2 \frac{\dot\psi^2}{\psi^2}} \notag \ ,
\end{gather}
\begin{gather}\label{eq_a}
 a_1=\frac{p-1}{\psi}\qua{\frac{\ddot \psi}{\psi}+ \frac{\dot \psi^2}{\psi^2}\ton{\frac{n(p-1)}{p(n-1)}-2}} \ .
\end{gather}
Note that here both $\psi$ and $\phi$ are defined as functions of $u(x)$.\\
Now we want to have two smooth positive functions $\psi$ and $\phi$ such that $a_3=0$ and $a_1$ and $a_2$ are strictly positive everywhere, so that
\begin{gather*}
 F(a_1 F+a_2)\leq 0
\end{gather*}
and necessarily $F$ is nonpositive at its point of maximum, so it is nonpositive on the whole manifold $M$.
\paragraph{\textsc{Coefficient $a_3$}}
Since $a_3$ is a function of $u(x)$, we can eliminate this dependence and rewrite $a_3$ as $a_3\circ u^{-1}$:
\begin{gather}\label{eq_cx}
a_3(s)=p\psi\qua{\frac{\lambda^2 \abs s^{2p-2}}{n-1} + \frac{n+1}{n-1}\frac{p-1}{p}\lambda s^{(p-1)} \dot \phi +\right.\\
+ \left.\frac{n (p-1)^2}{p^2(n-1)} \dot \phi^2 -\lambda (p-1)\phi \abs s^{p-2}-\frac{p-1}p \phi \ddot \phi}\notag \ ,
\end{gather}
where $s\in [-\xi,\xi\max\{u\}]$. Recall that $\phi=\dot w^p |_{w^{-1}(s)}$, so computing its derivatives it is important not to forget the derivative of $w^{-1}$, in particular
\begin{gather*}
 \dot \phi = p \abs{\dot w }^{p-2} \ddot w \ .
\end{gather*}
Remember that for a function of one variable, the $p$-Laplacian is $$\Delta_p w \equiv (p-1) \dot w^{p-2}\ddot w$$ so that we have
\begin{gather}\label{eq_phi}
 \frac{p-1}{p}\dot \phi = \Delta_p w \ \  ; \ \ \ \ \quad \frac{p-1} p \ddot \phi = \frac {d(\Delta_p w )} {dt} \frac{1}{\dot w} \ .
\end{gather}
With these substitutions, $a_3$ (or better $a_3\circ w:[w^{-1}(-\xi),w^{-1}(\xi \max\{u\})]\to \R$) can be written as
\begin{gather*}
 \frac{a_3}{p \psi}= \frac{n+1}{n-1} \lambda w^{(p-1)} \Delta_pw - \lambda(p-1) \abs w^{p-2} \dot w^p + \\
+\frac{\lambda^2\abs w^{2p-2}}{n-1}-\dot w^{p-1}\frac{d(\Delta_pw)}{dt}+ \frac{n}{n-1}(\Delta_pw)^2 \ .
\end{gather*}
Let $T$ be a solution to the ODE $\dot T=T^2/(n-1)$, i.e. either $T=0$ or $T=-\frac{n-1}{t}$ (note that from our point of view, there's no difference between $\frac{n-1}{t}$ and $\frac{n-1}{c+t}$, it's only a matter of shifting the variable $t$). A simple calculation shows that we can rewrite the last equation as
\begin{gather*}
\frac{a_3}{p\psi}=\frac{1}{n-1}\ton{\Delta_p w - T \dot w^{(p-1)} +\lambda w^{(p-1)}}\ton{n\Delta_p w + T \dot w^{(p-1)}+\lambda w^{(p-1)}}+\\
- \dot w ^{(p-1)}\frac d {dx} \ton{\Delta_p w - T \dot w^{(p-1)} +\lambda w^{(p-1)}} \ .
\end{gather*}
This shows that if our one dimensional model satisfies the ordinary differential equation
\begin{gather}\label{eq_1d}
 \Delta_p w = T\dot w ^{(p-1)}-\lambda w^{(p-1)} \ ,
\end{gather}
then $a_3=0$. Note that intuitively equation \eqref{eq_1d} is a sort of damped (if $T\neq 0$) $p$-harmonic oscillator. Remember that we are interested only in the solution on an interval where $\dot w >0$.

\vspace{0.5cm}
\paragraph{\textsc{Coefficients $a_1$ and $a_2$}}\label{sec_psi}
To complete the proof, we only need to find a strictly positive \[\psi \in C^2[-\xi,\xi\max\{u\}]\] such that both $a_1(u)$ and $a_2(u)$ are positive on all $M$. The proof is a bit technical, and relies on some properties of the model function $w$ that will be studied in the following section.

In order to find such a function, we use a technique similar to the one described in \cite[pag. 133-134]{new}. First of all, set by definition
\begin{gather}\label{ref_c}
X\equiv \lambda^{\frac{1}{p-1}} \frac {w(t)}{\dot w (t)}\ \ \ \ \  \psi(s)\equiv e^{\int h(s)} \ \ \  \ f(t)\equiv -h(w(t)) \dot w(t) \ ,
\end{gather}
so that
\begin{gather*}
 \dot f = - \dot h|_{w} \dot w ^2 - h|_{w} \ddot w = -\dot h|_{w} \dot w^2 - \frac {h|_{w}\dot w} {p-1} [T-X^{(p-1)}]=\\
=-\dot h|_{w} \dot w^2 + \frac {f} {p-1} [T-X^{(p-1)}] \ .
\end{gather*}
From equation \eqref{eq_a}, with our new definitions we have that
\begin{gather}\label{ref_d}
 \frac{a_1(w(t)) \psi(w(t))}{p-1}\dot w|_t ^2=
+\frac{f}{p-1}\qua{T-X^{(p-1)}} +f^2 \ton{\frac{p-n}{p(n-1)}}-\dot f\equiv \eta(f)-\dot f \ ,
\end{gather}
while if we use equations \eqref{eq_phi} and the differential equation \eqref{eq_1d} in equation \eqref{eq_b}, simple algebraic manipulations give
\begin{gather*}
\frac{a_2}{(p-1) \dot w^{p-2}} =-\frac{pT}{p-1} \ton{\frac{n}{n-1} T - X^{(p-1)}}+\\
-f^2 + f \qua{\ton{\frac{2n}{n-1} + \frac 1 {p-1}} T - \frac{p}{p-1} X^{(p-1)}} - \dot f\equiv\\
\equiv\beta(f) -\dot f \ .
\end{gather*}
The proof of the theorem now follows from Lemma \ref{lemma_etabeta}.

\end{proof}

Analyzing the case with boundary, the only difference in the proof of the gradient comparison is that the point $x_m$ may lie in the boundary of $M$, and so it is not immediate to conclude $\nabla F|_{x_m}=0$. However, once this is proved it is evident that $P^{II}_u F|_{x_m}\leq 0$ and the rest of proof proceeds as before. In order to prove that $x_m$ is actually a stationary point for $F$, the (nonstrict) convexity of the boundary is crucial. In fact we have
\begin{lemma}
 Let $M$ be as in Theorem \ref{grad_est}, but allow $M$ to have a (nonstrictly) convex $C^2$ boundary, and let $\Delta_p$ be the $p$-Laplacian with Neumann boundary conditions. Then, using the notation introduced above, if $\nabla u|_{x_m}\neq 0$
\begin{gather}
 \nabla F|_{x_m}=0 \ .
\end{gather}
\end{lemma}
\begin{proof}
We can assume that $x_m\in \partial M$, otherwise there would be nothing to prove. Let $\hat n$ be the outward normal derivative of $\partial M$.

Since $x_m$ is a point of maximum for $F$, we know that all the derivatives of $F$ along the boundary vanish, and that the normal derivative of $F$ is nonnegative
\begin{gather*}
 \ps{\nabla F}{\hat n}\geq 0 \ .
\end{gather*}
Neumann boundary conditions on $\Delta_p$ ensure that $\ps{\nabla u}{\hat n}=0$, and by direct calculation we have
\begin{gather*}
 \ps{\nabla F}{\hat n} = \qua{\ton{\abs{\nabla u}^p-\phi(u)}\dot \psi - \psi\dot \phi} \ps{\nabla u}{\hat n} +\\
+ p\psi(u)\abs{\nabla u}^{p-2} H_u(\nabla u,\hat n)=  p\psi(u)\abs{\nabla u}^{p-2} H_u(\nabla u,\hat n) \ .
\end{gather*}
Using the definition of second fundamental form $II(\cdot,\cdot)$, we can conclude
\begin{gather*}
 0\leq \ps{\nabla F}{\hat n} = p\psi(u)\abs{\nabla u}^{p-2} H_u(\nabla u,\hat n) = - p\psi(u)\abs{\nabla u}^{p-2} II(\nabla u,\nabla u)\leq 0 \ ,
\end{gather*}
and this proves the claim.
\end{proof}

Corollary \ref{cor_m} basically asserts that the fundamental estimate to prove the previous theorem is valid for any $n'\geq n$, so we can prove that
\begin{rem}\label{rem_m}
 The conclusions of Theorem \ref{grad_est} are still valid if we replace $n$ with any real $n'\geq n$.
\end{rem}
Note that while $n$ is the dimension of the Riemannian manifold under consideration, $n'$ does not represent any Riemannian entity.

\begin{rem}\label{rem_reg}
Before we end this section, we address the regularity issue in the gradient comparison theorem.
\end{rem}

 In the gradient comparison theorem, we assumed for simplicity $C^3$ regularity for $u$ around the point $x_m$. Since, as we have seen, we can assume without loss of generality that $\nabla u|_{x_m}\neq 0$, $C^3$ regularity is guaranteed if $p\geq 2$ or if $1<p<2$ and $u(x_m)\neq 0$.
 
 If we are in the case where $u(x_m)=0$ and $1<p<2$, $P^{II}_u(F)$ is not well defined. In particular, there are two terms diverging in this computation, one is $-\lambda \psi(u)\abs u ^{p-2} \abs{\nabla u}^{p}$ coming from $\psi(u)P^{II}_u (\abs{\nabla u}^2)$; the other is $-\psi(u) \ddot \phi(u) \abs{\nabla u}^p$, which comes from $-\psi(u) P^{II}_u (\phi(u))$.
 
 However, since $\nabla u|_{x}\neq 0$ in a neighborhood $U$ of $x_m$, we can still compute $P^{II}_u (F) $ on $U\setminus \{u=0\}$, which is an open dense set in $U$.
 
 If we assume that $\phi(u)=\dot w ^{p}|_{w^{-1}(u)}$, with $w$ satisfying \eqref{eq_1d}, it is easy to see that these two diverging terms precisely cancel each other in $U\setminus \{u=0\}$, and all the remaining terms in $P^{II}_u(F)$ are well-defined and continuous on $U$. Thus, even in this low-regularity case, the relation $P^{II}_u(F)|_{x_m}\leq 0$ is still valid, although the single terms $P^{II}_u(\abs{\nabla u}^p)$ and $P^{II}_u(\phi(u))$ are not well defined separately.

\section{One dimensional model}
This section contains the technical lemmas needed to study the properties of the solutions of the ODE
\begin{gather}\label{eq_1dm}
\begin{cases}
 \Delta_p w \equiv(p-1) \dot w^{p-2} \ddot w= T\dot w^{(p-1)} - \lambda w^{(p-1)}\\
w(a)=-1 \ \ \ \ \dot w(a)=0 
\end{cases}
\end{gather}
where either $T=-\frac{n-1}{t}$ or $T=0$. This second case has already been studied in Section \ref{sec_1d}, so we will concentrate on the first one.\\
To underline that this equation is to be considered on the real interval $[0,\infty)$ and not on the manifold $M$, we will denote by $t$ its independent variable. Notice that this ODE could be rewritten as
\begin{gather*}
 \frac d {dt} (t^{n-1} \dot w ^{(p-1)})+\lambda t^{n-1} w^{(p-1)}=0 \ ,
\end{gather*}
where $n\geq 2$ is the dimension of the manifold. Note that we define $u(x)=w(r(x))$ on $\R^n$, this equation characterizes the radial eigenfunction of the $p$-Laplacian. First of all we cite some known results on the solutions of this equation. 
\begin{teo}\label{teo_1dcont}
 If $a\geq0$, equation \eqref{eq_1dm} has always a unique solution which is of regularity $C^{1}(0,\infty)$ with $\dot w^{(p-1)} \in C^1(0,\infty)$, moreover if $a=0$ the solution belongs to $C^{0}[0,\infty)$. The solution depends continuously on the parameters in the sense of local uniform convergence of $w$ and $\dot w$ in $(0,\infty)$. Moreover every solution is oscillatory, meaning that there exists a sequence $t_k \to \infty$ such that $w(t_k)=0$.
\end{teo}
\begin{proof}
Existence, uniqueness and continuity with respect to the initial data and parameters is proved for example in \cite[Theorem 3, pag 179]{W}, and its oscillatory behaviour has been proved in the $n>p$ in \cite[Theorem 3.2]{kus_osc}, or in \cite[Theorems 2.2.11 and 2.3.4(i)]{dosly}, while the $n\leq p$ case is treated for example in \cite[Theorem 2.1]{wong_osc} and \cite[Theorem 2.2.10]{dosly}. Note that all these reference deal with much more general equations than the one we are interested in.
\end{proof}

In the following we will be interested only in the restriction of the solution $w$ to the interval $[a,b(a)]$, where $b(a)>a$ is the first point where $\dot w(b)=0$. It is easily seen that $\dot w\geq 0$ on $[a,b(a)]$, with strict inequality in the interior of the interval. Set by definition $t_0$ to be the only point in $[a,b(a)]$ such that $w(t_0)=0$.

First of all, we state and prove the lemma needed to complete the gradient comparison. Fix $a$ and the corresponding solution $w$, and define for simplicity on $(a,b)$
\begin{gather*}
 X(t)\equiv {\lambda}^{\frac 1 {p-1}}\frac{w(t)}{\dot w(t)} \ \ \quad \ T(t)= -\frac{n-1}{t} \ .
\end{gather*}
By direct calculation
\begin{gather}\label{eq_dxp}
 \frac{d}{dx} X^{(p-1)} = (p-1) \lambda^{\frac{1}{p-1}}\abs X ^{p-2} - T X^{(p-1)} + \abs{X}^{2(p-1)}\\
\dot X = \lambda^{\frac{1}{p-1}} - \frac 1 {p-1} T X + \frac 1 {p-1} \abs{X}^p \ .
\end{gather}
\begin{lemma}\label{lemma_etabeta}
Let $\eta(s,t)$ and $\beta(s,t)$ be defined by:
\begin{gather*}
 \eta(s,t)=\frac{s}{p-1}\qua{T-X^{(p-1)}} +s^2 \ton{\frac{p-n}{p(n-1)}} \ ,\\
\beta(s,t)=-\frac{pT}{p-1} \ton{\frac{n}{n-1} T - X^{(p-1)}}-s^2 +\\
+ s\qua{\ton{\frac{2n}{n-1} + \frac 1 {p-1}} T - \frac{p}{p-1} X^{(p-1)}} \ .
\end{gather*}
For every $\epsilon>0$, there exists a function $f:[a+\epsilon,b(a)-\epsilon]\to \R$ such that
\begin{gather}\label{eq_f}
 \dot f < \min\{\eta(f(t),t),\beta(f(t),t)\}
\end{gather}
\end{lemma}
\begin{proof}
We will prove that there exists a function $f:(a,b(a))\to \R$ which solves the ODE 
\begin{gather}\label{eq_fe}
 \begin{cases}
\dot f=\min\{\eta(f(t),t),\beta(f(t),t)\}\\
f(t_0)=\frac p {p-1} T_0 \ ,
 \end{cases}
\end{gather}
where we set $T_0=T(t_0)$. Then the lemma follows by considering the solution to
\begin{gather}\label{ref_e}
\begin{cases}
\dot f_\eta=\min\{\eta(f_\eta(t),t),\beta(f_\eta(t),t)\}-\eta\\
f_\eta(t_0)=\frac p {p-1} T(t_0)\equiv  \frac p {p-1} T_0 \ .
 \end{cases}
\end{gather}
Thanks to standard comparison theorems for ODE, if $\eta>0$ is small enough the solution $f_\eta$ is defined on $[a+\epsilon,b(a)-\epsilon]$ and satisfies the inequality \eqref{eq_fe}.

Consider that by Peano theorem there always exists a solution to \eqref{eq_f} defined in a neighborhood of $t_0$. We will show that this solution does not explode to infinity inside $(a,b)$, while we allow the solution to be infinity at the boundary of the interval. First of all note that for each $t\in (a,b(a))$
\begin{gather}
 \lim_{s\to \pm\infty} \min\{\eta(s,t),\beta(s,t)\}=-\infty \ .
\end{gather}
Then the solution $f$ is bounded from above in $(t_0,b)$ and bounded from below on $(a,t_0)$.

A simple calculation shows that
\begin{gather}\label{eq_dQ}
 \eta(f)-\beta(f)= \frac{p-1}p\frac{n}{n-1}(f-y_1)(f-y_2) \ ,
\end{gather}
where
\begin{gather*}
 y_1\equiv \frac{p}{p-1}\ton{T-\frac{n-1}n X^{(p-1)}} \ \ \ \ \ y_2\equiv \frac{p}{p-1} T \ .
\end{gather*}

Now we will prove that $f>y_1$ on $(t_0,b)$ and $f<y_1$ on $(a,t_0)$, and this will complete the proof of the lemma.
First we prove the inequality only in a neighborhood of $t_0$, i.e., we show that that there exists $\epsilon>0$ such that
\begin{gather}\label{eq_cfre}
 f(t)>y_1(t) \ \ \text{  for  } \ t_0<t<t_0+\epsilon \ ,\\
\notag f(t)<y_1(t) \ \ \text{  for  }\ t_0-\epsilon<t<t_0 \ .
\end{gather}
In fact, using the ODE \eqref{eq_f}, at $t_0$ we have
\begin{gather*}
 \dot f \vert_{t_0} = \frac{p}{(p-1)(n-1)} T_0^2 \ ,
\end{gather*}
while, where defined,
\begin{gather*}
\dot y_1 = \frac p {p-1} \qua{\frac{T^2}{n-1} - \lambda^{\frac 1 {p-1}}\frac{(n-1)(p-1)}{n} \abs{X}^{p-2} +\right.\\
+\left. \frac{n-1}{n} T X^{(p-1)} - \frac{n-1}{n} \abs{X}^{2(p-1)}} \ .
\end{gather*}
If $p=2$, $\dot f|_{t_0} -\dot y|_{t_0}>0 $, and if $p<2$
\begin{gather*}
 \lim_{t\to t_0} \dot f|_{t} -\dot y|_{t}=+\infty \ .
\end{gather*}
Thus, if $p\leq 2$, it is easy to conclude that \eqref{eq_cfre} holds. Unfortunately, if $p>2$, $y_1\in C^1((a,b))$ but $\dot f|_{t_0} -\dot y|_{t_0}=0$.

However, by equation \eqref{eq_dQ}, $\eta(y_1)=\beta(y_1)=\min\{\eta(y_1),\beta(y_1)\}$. In particular
\begin{gather}\label{ref_f}
 \eta(y_1)-\frac{p}{(p-1)(n-1)} T^2=-\frac{p(2p-1)}{(p-1)^2 n} T X^{(p-1)}+\\
\notag+ \frac{p^2(n-1)}{(p-1)^2 n^2}\abs X^{2(p-1)}=c_1 X^{(p-1)} + o(X^{(p-1)}) \ ,
\end{gather}
while
\begin{gather*}
 \dot y_1 -\frac{p}{(p-1)(n-1)} T^2 = -c_2 \abs X^{p-2}+ O(X^{(p-1)}) \ ,
\end{gather*}
where $c_2>0$. If follows that in a neighborhood of $t_0$, $y_1$ solves the differential inequality
\begin{gather*}
\begin{cases}
 \dot y_1\leq \min\{\eta(y_1),\beta(y_1)\}\\
y_1(t_0)=\frac{p}{p-1} T_0 
\end{cases}
\end{gather*}
and, applying a standard comparison theorem for ODE (see for example \cite[Theorem 4.1 in Chapter 3]{hart}), we can prove that the inequalities \eqref{eq_cfre} hold in a neighborhood of $t_0$.

To prove that they are valid on all $(a,b)$, suppose by contradiction that there exists some $t_1\in (a,t_0)$ such that $f(t_1)=y_1(t_1)$. The same argument works verbatim if $t_0< t_1 <b$.

Define $d(t)\equiv f(t)-y_1(t)$. By \eqref{eq_dQ}, $\dot f|_{t_1}=\eta(f(t_1),t_1)=\eta(y_1(t_1),t_1)$, which implies that
\begin{gather*}
 \dot d(t_1)= \frac{p(n-1)}{n} \lambda^{\frac 1 {p-1}}\abs X^{p-2}- \frac{p(n(p-1)+p)}{n(p-1)^2}T X^{(p-1)}+ \\
+\frac{(n-1)p(n(p-1)+p)}{n^2 (p-1)^2}\abs X^{2p-2}=\\
=p(n-1)\abs X^{p-2} \qua{\frac {\lambda^{\frac 1 {p-1}}} n+\frac{n(p-1)+p}{(p-1)^2 n^2}X\ton{X^{(p-1)}-\frac{n}{n-1}T}}\equiv \\
\equiv\frac{p(n-1)}{(p-1)^2 n^2}\abs X^{p-2}\kappa(t_1) \ .
\end{gather*}
where we set:
\begin{gather}\label{damnedH}
\kappa(x)\equiv n(p-1)^2\lambda^{\frac 1 {p-1}} + (n(p-1)+p)X\ton{X^{(p-1)}-\frac{n}{n-1}T} \ .
\end{gather}

Now we claim that $\kappa(t)$ is strictly positive for $t\neq t_0$, so that it is impossible for $d$ to be zero in a point different from $t_0$.

If $a>0$, it is evident that 
\begin{gather}
 \liminf_{t\to a,t_0,b} \kappa(t) >0 \ .
\end{gather}
If $a=0$ the same conclusion holds thanks to an approximation argument.

To show that $k(t)$ is positive everywhere, we argue by contradiction. Consider a point $z\in (t_0,b)$ where $\kappa(z)=0$ (the same argument works also if $z\in (a,t_0)$. At $z$ we have
\begin{gather*}
 X^{(p-1)}=-\frac{n(p-1)^2 \lambda^{\frac 1 {p-1}}}{(n(p-1)+p)X} + \frac{n}{n-1}T
\end{gather*}
and
\begin{gather*}
 \dot k = -\dot X \frac{n(p-1)^2 \lambda^{\frac 1 {p-1}}}{X}+ (n(p-1)+p)X\ton{\frac {d}{dt}X^{(p-1)}-\frac{nT^2}{(n-1)^2}} \ .
\end{gather*}
Using equation \eqref{eq_dxp} and some algebraic manipulations, we obtain
\begin{gather}\label{eq_kappa}
 \dot \kappa = -\frac{n (-1 + p)^2 p^2}{(n (-1 + p) + p) X}\lambda^{\frac{2}{p-1}} \ .
\end{gather}

This expression has a constant sign on $(t_0,b)$, and is never zero. For this reason, $z$ cannot be a minimum point for $k$, and so there exists a point $z'\in (z,b)$ such that $k(z')=0$ and $k(z)>0$ on $(z',b)$. Since $\dot k(z)$ and $\dot k(z')$ have the same sign, we have a contradiction.

\end{proof}

As will be clear later on, in order to obtain a sharp estimate on the first eigenvalue of the $p$-Laplacian we need to study the difference $\delta(a)=b(a)-a$ and find its minimum as a function of $a$. Note that if $T=0$, then the solution $w$ is invariant under translations and in particular $\delta(a)$ is constant and equal to $\frac{\pi_p}{\alpha}$, so we will restrict our study to the case $T\neq 0$. For ease of notation, we extend the definition of $\delta$ setting $\delta(\infty)=\frac{\pi_p}{\alpha}$.

In order to study the function $\delta(a)$, we introduce the \pf transformation (see \cite[section 1.1.3]{dosly} for a more detailed reference). Roughly speaking, the \pf transformation defines new variables $e$ and $\varphi$, which are the $p$-polar coordinates in the phase space of the solution $w$.

We set
\begin{gather}\label{ref_g}
  e(t)\equiv \ton{\dot w ^p + \alpha^p w^p}^{1/p}\ , \ \ \ \ \varphi(t)\equiv \operatorname{arctan_p} \ton{\frac{\alpha w}{\dot w}} \ .
\end{gather}
Recall that $\alpha= \ton{\frac{\lambda}{p-1}}^{1/p}$, so that
\begin{gather*}
 \alpha w = e \sinp (\varphi) \ , \ \ \dot w =e \cosp(\varphi) \ .
\end{gather*}
Differentiating, simplifying and using equation \eqref{eq_1dm}, we get the following differential equations for $\varphi$ and $e$
\begin{gather}\label{eq_pf}
 \dot \varphi = \alpha -\frac{T(t)}{p-1}\cosp^{p-1} (\varphi)\sinp(\varphi)=\alpha +\frac{n-1}{(p-1)t}\cosp^{p-1} (\varphi)\sinp(\varphi) \\
\frac{\dot e}{e} =\frac {T(t)}{p-1}\cosp^{p} (\varphi)=-\frac{n-1}{(p-1)t}\cosp^{p} (\varphi)\notag \ .
\end{gather}
Rewritten in this form, it is quite straightforward to prove existence, uniqueness and continuous dependence of the solutions of the ODE \eqref{eq_1dm} at least in the case $a>0$. Moreover, note that the derivative of $\varphi$ is strictly positive. Indeed, this is obviously true at the points $a, \ b(a)$ where $\dot w=0$ implies $\cosp(\varphi)=0$, while at the points where $\dot \varphi =0$ we have by substitution that $\ddot \varphi= \frac{\alpha} t $, which is always positive, so it is impossible that $\dot \varphi$ reaches zero. Moreover, a slight modification of this argument shows that $\dot \varphi$ is in fact bounded from below by $\frac{\alpha} n$. Indeed, consider by contradiction a point where $\dot \varphi=\frac{\alpha}{n} - \epsilon$, then
\begin{gather*}
 \ddot \varphi = \frac 1 t \ton{-\frac{n-1}{(p-1)t}\cosp^{p-1}(\varphi)\sinp(\varphi)+ \frac{n-1}{p-1}(1-p\abs{\sinp(\varphi)}^p)\dot \varphi}\geq \frac n t \epsilon \ .
\end{gather*}
Since $\dot \varphi(a)=\alpha$, it is evident that such a point cannot exist. This lower bound on $\dot \phi$ proves directly the oscillatory behaviour of the solutions of ODE \eqref{eq_1dm}.

Note that, for every solution, $e$ is decreasing (strictly if $T\neq 0$), which means that the absolute value of local maxima and minima decreases as $t$ increases.

Now we are ready to prove the following lemma
\begin{teo}\label{teo_delta_comp}
 For any $n>1$, the difference $\delta(a)$, which is a continuous function on $[0,+\infty)$, is always strictly greater than ${\pi_p}/\alpha$, i.e. the difference $\delta(a)$ in the case $T=0$. Moreover, let $m(a)\equiv w(b(a))$, then for every $a\in [0,\infty)$
\begin{gather*}
 \lim_{a\to \infty} \delta(a)=\frac{\pi_p}{\alpha} =\delta(\infty) \ , \ \ \ \ \ m(a)<1 \ , \ \ \ \ \lim_{a\to \infty} m(a)=1 \ .
\end{gather*}
\end{teo}
\begin{proof}
Continuity follows directly from Theorem \ref{teo_1dcont}. To prove the estimate, we rephrase the question in the following way: consider the solution $\varphi$ of the initial value problem
\begin{gather*}
\begin{cases}
  \dot \varphi = \alpha + \frac{n-1}{(p-1)t}\cosp^{p-1} (\varphi)\sinp(\varphi)\\
\varphi(a)=-\frac{\pi_p}{2} 
\end{cases}
\end{gather*}
then $b(a)$ is the first value $b>a$ such that $\varphi(b)=\frac{\pi_p}2$. Denote $t_0\in (a,b)$ the only value where $\varphi(t_0)=0$, then it is easily seen that
\begin{gather*}
 \frac{n-1}{(p-1)t}\cosp^{p-1} (\varphi)\sinp(\varphi)\leq \frac{n-1}{(p-1)t_0}\cosp^{p-1} (\varphi)\sinp(\varphi) \ ,
\end{gather*}
so that $\varphi$ satisfies
\begin{gather*}
0<  \dot \varphi \leq \alpha + \gamma\cosp^{p-1} (\varphi)\sinp(\varphi) \ ,
\end{gather*}
where $\gamma=\frac{n-1}{(p-1)t_0}$. By a standard comparison theorem for ODE, $\varphi\leq \psi$ on $[a,b]$, where $\psi$ is the solution of the initial value problem
\begin{gather*}
\begin{cases}
  \dot \psi = \alpha + \gamma\cosp^{p-1} (\psi)\sinp(\psi)\\
\psi(a)=-\frac{\pi_p}{2} 
\end{cases}
\end{gather*}
We can solve explicitly this ODE via separation of variables. Letting $c(a)$ be the first value $c>a$ such that $\psi(c)=\pi_p/2$, we have
\begin{gather*}
 c(a)-a=\int_{-\frac {\pi_p} 2}^{\frac {\pi_p} 2} \frac{d\psi}{\alpha + \gamma\cosp^{p-1} (\psi)\sinp(\psi)} \ .
\end{gather*}
Applying Jansen's inequality, and noting that $\cosp^{p-1}(\psi)\sinp(\psi)$ is an odd function, we obtain the estimate
\begin{gather}\label{eq_jan}
\frac{c(a)-a}{\pi_p}= \frac 1 {\pi_p} \int_{-\pi_p/2}^{\pi_p/2} \frac{d\psi}{\alpha+\gamma \cosp^{p-1} (\psi)\sinp(\psi)}\geq\\
\geq \qua{\frac 1 {\pi_p}\int_{-\pi_p/2}^{\pi_p/2} \ton{\alpha+\gamma \cosp^{p-1} (\psi)\sinp(\psi)}d\psi}^{-1} = \frac{1}{\alpha} \ . 
\end{gather}
Note that the inequality is strict if $\gamma\neq0$, or equivalently if $T\neq 0$.

Since $\varphi\leq \psi$, it is easily seen that $b(a)\geq c(a)$, and we can immediately conclude that $\delta(a) \geq \pi_p/\alpha$ with equality only if $a=\infty$.

The behaviour of $\delta(a)$ as $a$ goes to infinity is easier to study if we perform a translation of the $t$ axis, and study the equation
\begin{gather*}
\begin{cases}
  \dot \varphi = \alpha + \frac{n-1}{(p-1)(t+a)}\cosp^{p-1} (\varphi)\sinp(\varphi)\\
\varphi(0)=-\frac{\pi_p}{2} 
\end{cases}
\end{gather*}
Continuous dependence on the parameters of the equation allows us to conclude that if $a$ goes to infinity, then $\varphi$ tends to the affine function $\varphi_0(t)=-\frac{\pi_p}{2} + \alpha {t}$ in the local $C^1$ topology. This proves the first claim. As for the statements concerning $m(a)$, note that the inequality $m(a)<1$ follows directly from the fact that $\frac{\dot e }{e}<0$ if $T\neq 0$. Moreover we can see that $m(a)=\tilde w (\delta(a))$, where $\tilde w$ is the solution of
\begin{gather}
\begin{cases}
 \Delta_p \tilde w =(p-1) \tilde {\dot w}^{p-2} \tilde {\ddot w}= -\frac{n-1} {x+a} \tilde {\dot w}^{(p-1)} - \lambda \tilde  w^{(p-1)}\\
\tilde w(0)=-1 \ \ \ \ \tilde {\dot w}(0)=0 \ .
\end{cases}
\end{gather}
The function $\tilde w$ converges locally uniformly to $\sinp(\alpha t -\frac{\pi_p}{2})$ as $a$ goes to infinity, and since $\delta(a)$ is bounded from above, it is straightforward to see that $\lim_{a\to \infty} m(a)=1$.
\end{proof}

As an immediate consequence of the above theorem, we have the following important
\begin{cor}\label{cor_delta}
 The function $\delta(a):[0,\infty]\to \R^+$ is continuous and
\begin{align}
 \delta(a)>\frac{\pi_p} \alpha \ \ \ \ \ &\text{ for  } a \in [0,\infty) \ ,\\
 \delta(a)=\frac{\pi_p} \alpha \ \ \ \ \ &\text{ for  } a =\infty \ . 
\end{align}
\end{cor}
Recall that $a=\infty$ if and only if $m(a)=1$, and also $\delta(a)=\frac{\pi_p} \alpha$ if and only if $m(a)=1$.

\section{Volume estimates}\label{sec_max}
The next comparison theorem will allow us to compare the maxima of eigenfunctions with the maxima of the model functions, so it is essential for the proof of the main theorem. We begin with some definitions.
\begin{deph}
 Given $u$ eigenfunction and $w$ as before, let $t_0\in (a,b)$ be the unique zero of $w$ and let $g\equiv w^{-1}\circ u$. We define the measure $m$ on $[a,b]$ by
\begin{gather*}
 m(A)\equiv \V(g^{-1}(A)) \ ,
\end{gather*}
where $\V$ is the Riemannian measure on $M$. Equivalently, for any bounded measurable $f:[a,b]\to \R$, we have
\begin{gather*}
 \int_a^b f(s) dm(s) = \int_M f(g(x))d\V(x) \ . 
\end{gather*}

\end{deph}
\begin{teo}\label{teo_comp}
 Let $u$ and $w$ be as above, and let
\begin{gather*}
 E(s)\equiv -\operatorname{exp}\ton{\lambda\int_{t_0}^s \frac{w^{(p-1)}}{\dot w^{(p-1)}} dt}\int_a ^s{w(r)}^{(p-1)} dm(r)
\end{gather*}
Then $E(s)$ is increasing on $(a,t_0]$ and decreasing on $[t_0,b)$.
\end{teo}
Before the proof, we note that this theorem can be rewritten in a more convenient way. Consider in fact that by definition
\begin{gather*}
 \int_{a}^s \pw(r)\ dm(r)= \int_{\{u\leq w(s) \}}u(x)^{(p-1)}\ d\V(x) \ .
\end{gather*}
Moreover, note that the function $w$ satisfies
\begin{gather*}
 \frac d {dt}(t^{n-1}\pdw)=-\lambda t^{n-1}\pw \ , \\
 -\lambda\frac{\pw}{\pdw}=
\frac d {dt} \log(t^{n-1}\pdw) \ ,
\end{gather*}
and therefore
\begin{gather*}
 -\lambda\int_a ^s \pw(t)t^{n-1}dt = s^{n-1}\pdw (s) \ , \\
\operatorname{exp}\ton{\lambda \int_{t_0}^s \frac{\pw}{\pdw} dt }=\frac{t_0^{n-1}\pdw(t_0)}{s^{n-1}\pdw(s)} \ .
\end{gather*}
Thus, the function $E(s)$ can be rewritten as
\begin{gather*}
 E(s)=C\frac{\int_a ^s\pw(t) \ dm(r)}{\int_a ^s \pw(t)\ t^{n-1}dt}=C\frac{\int_{\{u\leq w(s)\}}{u(x)}^{(p-1)}\ d\V(x)}{\int_a ^s {w(t)}^{(p-1)}\ t^{n-1}dt} \ ,
\end{gather*}
where $\lambda C^{-1}=t_0^{n-1}\pdw(t_0)$, and the previous theorem can be restated as follows.
\begin{teo}\label{int_comp}
 Under the hypothesis of the previous theorem, the ratio
\begin{gather*}
 E(s)=\frac{\int_a ^s\pw(r) \ dm(r)}{\int_a ^s \pw(t)\ t^{n-1}dt}=\frac{\int_{\{u\leq w(s)\}}{u(x)}^{(p-1)}\ d\V(x)}{\int_a ^s {w(t)}^{(p-1)}t^{n-1} dt}
\end{gather*}
is increasing on $[a,t_0]$ and decreasing on $[t_0,b]$.
\end{teo}
\begin{proof}[\textsc{Proof of Theorem \ref{teo_comp}}]
 Chose any smooth nonnegative function $H(s)$ with compact support in $(a,b)$, and define $G:[-1,w(b)]\to \R$ in such a way that
\begin{gather*}
 \frac d {dt} \qua{{G(w(t))}^{(p-1)}} =H(t) \quad \ \ G(-1)=0 \ .
\end{gather*}
It follows that
\begin{gather*}
 \pG(w(t))= \int_a ^t H(s) ds \ \ \ \ (p-1) \abs{G(w(t))}^{p-2} \dot G(w(t)) \dot w (t) = H(t) \ .
\end{gather*}
Then choose a function $K$ such that $(t K(t))'=K(t)+t \dot K(t)=G(t)$. By the chain rule we obtain
\begin{gather*}
 \Delta_p (uK(u))=G^{(p-1)} (u) \Delta_p(u) + (p-1)\abs{G(u)}^{p-2} \dot G(u) \abs{\nabla u}^p \ .
\end{gather*}
Using the weak formulation of the divergence theorem, it is straightforward to verify that \[\int_M \Delta_p(uK(u))d \V =0\], and so we get
\begin{gather*}
\frac{\lambda  }{p-1} \int_M \pu \pG(u) d\V(x)= \int_M \abs{G(u)}^{p-2} \dot G (u) \abs{\nabla u}^{p} d\V \ .
\end{gather*}
Applying the gradient comparison theorem (Theorem \ref{grad_est}), noting that we consider only $\lambda>0$, we have
\begin{gather*}
\frac{\lambda  }{p-1} \int_M \pu \pG (u)d\V(x)\leq \int_M \abs{G(u)}^{p-2} \dot G (u) (\dot w \circ w^{-1}(u)) ^p d\V \ .
\end{gather*}
By definition of $d m $, the last inequality can be written as
\begin{gather*}
\frac{\lambda  }{p-1} \int_a^b \pw(s) \pG(w(s)) dm(s)\leq \\
\leq\int_a^b \abs{G(w(s))}^{p-2} \dot G (w(s)) (\dot w(s)) ^p dm(s) \ ,
\end{gather*}
and recalling the definition of $G$ we deduce that
\begin{gather*}
 \lambda \int_a^b \pw(s) \ton{\int_a^s H(t) dt}dm(s)= \lambda \int_a^b  \ton{\int_s^b \pw(t) dm(t)} H(s)ds \leq \\
\leq \int_a^b H(s) \pdw(s)\ dm(s) \ .
\end{gather*}
Since $\int_a^b \pw(t) dm(t)=0$, we can rewrite the last inequality as
\begin{gather*}
 \int_a^b  H(s) \qua{-\lambda\int_a^s \pw(t) dm(t)} ds \leq \int_a^b H(s) \pdw(s)\ dm(s) \ .
\end{gather*}
Define the function $A(s)\equiv -\int_a^s \pw(r) dm(r)$. Since the last inequality is valid for all smooth nonnegative function $H$ with compact support, then
\begin{gather*}
\pdw(s) dm(s)-\lambda A(s)ds \geq 0 
\end{gather*}
in the sense of distributions, and therefore the left hand side is a positive measure. In other words, the measure $\lambda Ads+ \frac{\pdw}{\pw} dA$ is nonpositive. Of if we multiply the last inequality by $\frac{\pw}{\pdw}$, and recall that $w\geq 0$ on $[t_0,b)$ and $w\leq 0$ on $(a,t_0]$, we conclude that the measure 
\begin{gather*}
 \lambda  \frac{\pw}{\pdw}Ads+  dA
\end{gather*}
is nonnegative on $(a,t_0]$ and nonpositive on $[t_0,b)$, or equivalently the function
\begin{gather*}
 E(s)= A(s) \operatorname{exp}\ton{\lambda\int_{t_0}^s \frac{\pw}{\pdw}(r) dr}
\end{gather*}
is increasing on $(a,t_0]$ and decreasing on $[t_0,a)$.
\end{proof}

Before we state the comparison principle for maxima of eigenfunctions, we need the following lemma. The definitions are consistent with the ones in Theorem \ref{grad_est}.
\begin{lemma}\label{lemma_radepsilon}
 For $\epsilon$ sufficiently small, the set $u^{-1}[-1,-1+\epsilon)$ contains a ball of radius $r=r_\epsilon$, which is determined by
\begin{gather*}
r_\epsilon=w ^{-1}(-1+\epsilon)-a \ .
\end{gather*}
\begin{proof}
 This is a simple application of the gradient comparison principle (Theorem \ref{grad_est}). Let $x_0$ be a minimum point of $u$, i.e. $u(x_0)=-1$, and let $\bar x$ be another point in the manifold. Let $\gamma:[0,l]\to M$ be a unit speed minimizing geodesic from $x_0$ to $\bar x$, and define $f(t)\equiv u(\gamma(t))$. It is easy to see that
\begin{gather}\label{ref_h}
\abs{\dot f (t)}=\abs{\ps{\nabla u |_{\gamma(t)}}{\dot \gamma(t)}}\leq \abs{\nabla u |_{\gamma(t)}}\leq \dot w |_{w^{-1}(f(t))}  \ .
\end{gather}
Since
\begin{gather*}
\frac d {dt} w^{-1}(f(t))\leq 1 \ ,
\end{gather*}
we have that $a\leq w^{-1}(f(t))\leq a+t$, and since $\dot w$ is increasing in a neighborhood of $a$, we can deduce that
\[\dot w |_{w^{-1}f(t)}  \leq \dot w |_{a+t} \ .\]
By the absolute continuity of $u$ and $\gamma$, we can conclude that
\begin{gather*}
 \abs{f(t)+1}\leq\int_0^t  \dot w |_{a+s} ds = (w(a+ t)+1) \ .
\end{gather*}
This means that if $l=d(x_0,\bar x)< w ^{-1}(-1+\epsilon)-a$, then $u(\bar x)< -1+\epsilon$.
\end{proof}

\end{lemma}

And now we are ready to prove the comparison theorem. 
\begin{teo}
If $u$ is an eigenfunction on $M$ such that $\min\{u\}=-1=u(x_0)$ and $\max\{u\}\leq m(0)=w(b(0))$, then for every $r>0$ sufficiently small, the volume of the ball centered at $x_0$ and of radius $r$ is controlled by
\begin{gather*}
 \V(B(x_0,r))\leq c r^n \ .
\end{gather*}

\end{teo}
\begin{proof}
 Denote by $\nu$ the measure $t^{n-1}dt$ on $[0,\infty)$. For $k\leq -1/2^{p-1}$, applying Theorem \ref{int_comp} we can estimate
\begin{gather*}
 \V(\{u\leq k\})\leq -2 \int_{\{u\leq k\}}  u^{(p-1)} d\V \leq\\
\leq -2 C \int_{\{w\leq k\}} w^{(p-1)} d\nu\leq 2C\nu(\{w\leq k\}) \ .
\end{gather*}
If we set $k=-1+\epsilon$ for $\epsilon$ small enough, it follows from Lemma \ref{lemma_radepsilon} that there exist constants $C$ and $C'$ such that
\begin{gather*}
\V(B(x_0,r_\epsilon))\leq  \V(\{u\leq k\})\leq \\
\leq2C\nu(\{w\leq -1+\epsilon\})= 2C \nu([0,r_\epsilon])= C' r_\epsilon^n \ .
\end{gather*}

\end{proof}

\begin{cor}\label{cor_max}
\rm{ As a corollary, we get that $\max\{u\}\geq m$. In fact, suppose by contradiction that $\max\{u\}<m$. Then, by the continuous dependence of solutions of ODE \eqref{eq_1dm} on the parameters, there exists $n'>n$ ($n'\in \R$) such that $\max\{u\}\leq m(n')$, i.e., there exists an $n'$ such that the solution $w'$ to the ode
\begin{gather*}
 \begin{cases}
 (p-1) \dot w'^{p-2} \ddot w' - \frac{n'-1}{t} \dot w'^{(p-1)} + \lambda w'^{(p-1)}=0\\
w'(0)=-1 \\ \dot w' (0)=0 
 \end{cases}
\end{gather*}
has a first maximum which is still greater than $\max\{u\}$. By Remark \ref{rem_m}, the gradient estimate $\abs{\nabla u}\leq \dot w'|_{w'^{-1}(u)}$ is still valid and so is also the volume comparison. But this is contradicts the fact that the dimension of the manifold is $n$. In fact one would have that for small $\epsilon$ (which means for $r_\epsilon$ small)  $\V(B(x_0,r_\epsilon))\leq c r_\epsilon^{n'}$. Note that the argument applies even in the case where $M$ has a $C^2$.}
\end{cor}

\section{Sharp estimate}
Now we are ready to state and prove the main theorem.
\begin{teo}\label{teo_main}
 Let $M$ be a compact Riemannian manifold with nonnegative Ricci curvature, diameter $d$ and possibly with convex boundary. Let $\lambda_p$ be the first nontrivial (=nonzero) eigenvalue of the $p$-Laplacian (with Neumann boundary condition if necessary). Then the following sharp estimate holds:
\begin{gather*}
 \frac{\lambda_p}{p-1} \geq \frac{\pi_p^p}{d^p} \ .
\end{gather*}
Moreover a necessary (but not sufficient) condition for equality to hold is that $\max\{u\}=-\min\{u\}$.
\end{teo}
\begin{proof}
 First of all, we rescale $u$ in such a way that $\min\{u\}=-1$ and $0<\max\{u\}=k\leq 1$. Given a solution to the differential equation \eqref{eq_1dm}, let $m(a)\equiv w(b(a))$ the first maximum of $w$ after $a$. We know that this function is a continuous function on $[0,\infty)$, and
\begin{gather*}
 \lim_{a\to \infty} m(a) =1 \ .
\end{gather*}
By Corollary \ref{cor_max}, $k\geq m(0)$. This means that for every eigenfunction $u$, there exists $a$ such that $m(a)=k$. If $k=1$, then $k=m(\infty)$.

We can rephrase this statement as follows: for any eigenfunction $u$, there exists a model function $w$ such that $\min\{u\}=\min\{w\}=-1$ and $0<\max\{u\}=\max\{w\}=k\leq 1$. Once this statement is proved, the eigenvalue estimate follows easily. In fact, consider a minimum point $x$ and a maximum point $y$ for the function $u$, and consider a unit speed minimizing geodesic (of length $l\leq d$) joining $x$ and $y$. Let $f(t)\equiv u(\gamma(t))$, and consider the subset $I$ of $[0,l]$ with $\dot f \geq0$. Then changing variables we get
 \begin{gather*}
  d\geq \int_0 ^l dt \geq \int_I dt \geq \int_{-1}^k \frac{dy}{\dot f (f^{-1}(y))}\geq\int_{-1}^k \frac{dy}{\dot w (w^{-1}(y))}=\\
=\int_a^{b(a)} 1 dt = \delta(a)\geq \frac{\pi_p}{\alpha} \ ,
 \end{gather*}
where the last inequality is proved in Corollary \ref{cor_delta}. This yields to
\begin{gather*}
 \frac{\lambda}{p-1}\geq \frac{\pi_p^p}{d^p} \ .
\end{gather*}
Note that by Corollary \ref{cor_delta}, for any $a$, $\delta(a)\geq \frac{\pi_p}{\alpha}$ and equality holds only if $a=\infty$, i.e. only if $\max\{u\}=-\min\{u\}=\max\{w\}=-\min\{w\}$.
\end{proof}

\begin{rem}\label{rem_=}\rm{
 Note that $\max\{u\}=\max\{w\}$ is essential to get a sharp estimate, and it is the most difficult point to achieve. Analyzing the proof of the estimate in \cite{hui} with the tools developed in this article, it is easy to realize that in some sense the only model function used in \cite{hui} is $\sinp(\alpha x)$, which leads to $\phi(u)=\frac{\lambda}{p-1}(1-\abs u^p)$. Since the maximum of this model function is $1$, which in general is not equal to $\max\{u\}$, the last change of variables in the proof does not hold. Nevertheless $\max\{u\}>0$, and so one can estimate that
\begin{gather*}
 d\geq \int_{-1}^k \frac{dy}{\dot w (w^{-1}(y))} > \int_{-1}^0 \frac{dy}{\dot w (w^{-1}(y))}= \int_{-1}^0 \frac{dy}{\alpha (1-y^p)^{1/p}}= \frac{\pi_p}{2 \alpha} \ ,
\end{gather*}
which leads to
\begin{gather*}
 \frac{\lambda}{p-1}>\ton{\frac{\pi_p}{2d}}^p \ .
\end{gather*}

}
\end{rem}

\section{Characterization of equality}
In this section we characterize the equality in the estimate just obtained, and prove that equality can be achieved only if $M$ is either a one dimensional circle or a segment.

In \cite{Hang}, this characterization is proved for $p=2$ to answer an open problem raised by T. Sakai in \cite{saka}. Unfortunately, this proof relies on the properties of the Hessian of a 2-eigenfunction, which are not easily generalized for generic $p$. In particular, any 2-eigenfunction is smooth everywhere and at a minimum point its Hessian is positive semidefinite. Moreover, $H(v,v)$ is the second derivative of the function $f(t)=u(\gamma(t))$ where $\gamma$ is the geodesic such that $v=\dot \gamma$. All these properties are essential for the proof given in \cite{Hang} to work, and the lack of them forces us to choose another way to prove the characterization.

Before we prove the characterization theorem, we need the following lemma, which is similar in spirit to \cite[Lemma 1]{Hang}.
\begin{lemma}\label{lemma_=}
 Assume that all the assumptions of Theorem \ref{teo_main} are satisfied and assume also that equality holds in the sharp estimate. Then we can rescale $u$ in such a way that $-\min\{u\}=\max\{u\}=1$, and in this case $e^p=\abs{\nabla u}^p+\frac \lambda {p-1}\abs{u}^p$ is constant on the whole manifold $M$, in particular $e^p=\frac \lambda {p-1}$. Moreover, all integral curves of the vector field $X\equiv \frac{\nabla u}{\abs {\nabla u}}$ are minimizing geodesics on the open set $E\equiv\{\nabla u\neq 0\}=\{u\neq \pm1\}$ and for all geodesics $\gamma$, $\ps{\dot \gamma}{\frac{\nabla u}{\abs{\nabla u}}}$ is constant on each connected component of $\gamma^{-1}(E)$.
\end{lemma}
\begin{proof}
Using the model function $w(t)=\sinp(\alpha t)$ in the gradient comparison, we know that 
\begin{gather*}
 \abs{\nabla u}^p\leq \abs{\dot w }|_{w^{-1} (u)}^p= \frac{\lambda}{p-1} (1-\abs{u}^p) \ ,
\end{gather*}
so that $e^p\leq \frac{\lambda}{p-1}$ everywhere on $M$. Let $x$ and $y$ be a minimum and a maximum point of $u$ respectively, and let $\gamma$ be a unit speed minimizing geodesic joining $x$ and $y$. Define $f(t)\equiv u(\gamma(t))$. Following the proof of Theorem \ref{teo_main}, we know that
\begin{gather*}
 d\geq \int_{-1}^1 \frac{ds}{\dot f (f^{-1}(s))}\geq \int_{-1}^1 \frac{ds}{\dot w (w^{-1}(s))} = \frac{\pi_p}{\alpha} \ .
\end{gather*}
By the equality assumption, $\alpha d = \pi_p$, and this forces $\dot f(t) = \dot w |_{w^{-1} f(t)}$. So, up to a translation in the domain of definition, $f(t)=w(t)$, and on the curve $\gamma$ $e^p|_{\gamma}=\frac \lambda{p-1}$.

Now the statement of the lemma is a consequence of the strong maximum principle (see for example \cite[Theorem 3.5 pag 34]{GT}). Indeed, consider the operator
\begin{gather*}
 L(\phi)=P_u(\phi) -\frac{(p-1)^2}{p \abs{\nabla u}^2}\ps{\nabla \ton{\abs{\nabla u}^p -\frac{\lambda}{p-1}u^{p}}}{\nabla \phi}+\\
+(p-2)^2\abs{\nabla u}^{p-4}\pst{\nabla u}{H_u}{\nabla \phi - \frac{\nabla u}{\abs{\nabla u}}\ps{\frac{\nabla u}{\abs{\nabla u}}}{\nabla \phi}} \ .
\end{gather*}
The second order part of this operator is $P^{II}_u$, so it is locally uniformly elliptic in the open set $E\equiv \{\nabla u \neq 0\}$, while the first order part (which plays no role in the maximum principle) is designed in such a way that $L(e^p)\geq 0$ everywhere. In fact, after some calculations we have that
\begin{gather}\label{eq_R}
 \frac {L(e^p)}{p\abs{\nabla u}^{2p-4}} = \ton{\abs{H_u}^2 - \frac{\abs{H_u(\nabla u)}^2}{\abs{\nabla u}^2}}+ \Ric(\nabla u,\nabla u)\geq 0 \ .
\end{gather}
Then by the maximum principle the set $\{e^p=\frac{\lambda}{p-1}\}$ is open and close in $Z\equiv \{u\neq \pm1\}$, so it contains the connected component $Z_1$ containing $\gamma$.

Let $Z_2$ be any other connected component of $Z$ and choose $x_i \in Z_i$ with $u(x_i)=0$ for $i=1,2$. Let $\sigma$ a unit speed minimizing geodesic joining $x_1$ and $x_2$. Necessarily there exists $\bar t$ such that $\sigma(\bar t)\subset u^{-1}\{-1,1\}$, otherwise $Z_1=Z_2$. Without loss of generality, let $u(\bar t)=1$ and define $f(t)\equiv u(\sigma(t))$. Arguing as before we can conclude
\begin{gather*}
 d\geq \int_0^l dt= \int_0^{\bar t} dt + \int_{\bar t }^l dt \geq \int_{I_1}dt +\int_{I_2} dt \geq \\
\int_0^1 \frac{dy}{\dot f (f^{-1}(y))}-\int_{0}^{1} \frac{dy}{\dot f (f^{-1}(y))}\geq 2\int_{0}^1 \frac{dy}{\dot w (w^{-1}(y))} \geq \frac{\pi_p}{\alpha} \ ,
\end{gather*}
where $I_1\subset[0,\bar t]$ is the subset where $\dot f >0$ and $I_2\subset[\bar t ,l]$ is where $\dot f<0$. The equality assumption forces $\alpha d =\pi_p$ and so $\dot f(f^{-1}(t))=\dot w (w^{-1}(t))$ a.e. on $[0,\bar t]$ and $\dot f(f^{-1}(t))=-\dot w (w^{-1}(t))$ a.e. on $[\bar t,l]$, which implies that, up to a translation in the domain of definition,  $f(t)=w(t)=\sinp(\alpha t)$. This proves that for any connected component $Z_2$, there exists a point inside $Z_2$ where $e^p=\frac{\lambda}{p-1}$, and by the maximum principle $e^p=\frac{\lambda}{p-1}$ on all $Z$.
This also proves that $E=Z$. Moreover, for equality to hold in \eqref{eq_R}, the Ricci curvature has to be identically equal to zero on $Z$ and
\begin{gather}\label{eq_H}
 \abs{H_u}^2 =\frac{\abs{H_u(\nabla u)}^2}{\abs{\nabla u}^2} \ .
\end{gather}
Now the fact that $e^p$ is constant implies by differentiation that where $\nabla u\neq0$, i.e. on $Z$, we have
\begin{gather*}
 \abs{\nabla u}^{p-2} H_u(\nabla u) = - \frac{\lambda}{p-1}u^{p-1} \nabla u \ ,
\end{gather*}
and so
\begin{gather*}
 \abs{\nabla u }^{p-2} \pst{X}{H_u}{X}= - \frac{\lambda}{p-1}u^{p-1} \ .
\end{gather*}
This and equation \eqref{eq_H} imply that on $Z$
\begin{gather}\label{eq_HH}
\abs{\nabla u}^{p-2} H_u = - \frac{\lambda}{p-1} u^{p-1} X^\star \otimes X^\star \ .
\end{gather}
Now a simple calculation shows that
\begin{gather*}
 \nabla_X X = \frac 1 {\abs{\nabla u}}\nabla_{\nabla u} \frac{\nabla u }{\abs{\nabla u}}= \frac{1}{\abs{\nabla u}} (H_u (X)- H_u(X,X)X)=0 \ .
\end{gather*}
Which proves that integral curves of $X$ are geodesics. The minimizing property follows from an adaptation of the proof of Theorem \ref{teo_main}.

As for the last statement, we have
\begin{gather*}
 \frac d {dt} \ps{\dot \gamma}{\frac{\nabla u}{\abs{\nabla u}}} = \frac{1}{\abs{\nabla u}} H_u(\dot \gamma,\dot \gamma)-\ps{\dot \gamma}{\nabla u} \frac{1}{\abs{\nabla u}^3} \pst{\nabla u}{H_u}{\dot \gamma}
\end{gather*}
and, by equation \eqref{eq_HH}, the right hand side is equal to $0$ where $\nabla u\neq 0$.

\end{proof}

As promised, now we are ready to state and prove the characterization.
\begin{teo}\label{teo_=}
 Let $M$ be a compact Riemannian manifold with $\Ric \geq 0$ and diameter $d$ such that
\begin{gather}\label{eq_=}
 \frac{\lambda_p}{p-1} = \ton{\frac{\pi_p}{d}}^p \ .
\end{gather}
If $M$ has no boundary, then it is a one dimensional circle; if $M$ has boundary then it is a one dimensional segment.
\end{teo}
\begin{proof}

We will prove the theorem studying the connected components of the set $N=u^{-1}(0)$, which, according to Lemma \ref{lemma_=}, is a regular submanifold. We divide our study in two cases, and we will show that in both cases $M$ must be a one dimensional manifold (with or without boundary).
\paragraph{\textsc{Case 1, $N$ has more than one component}}Suppose that $N$ has more than one connected component. Let $x$ and $y$ be in two different components of $N$ and let $\gamma$ be a unit speed minimizing geodesic joining them.  Since $\ps{\dot \gamma}{\frac{\nabla u}{\abs{\nabla u}}}$ is constant on $E$, either $\gamma(t)=0$ for all $t$, which is impossible since by assumption $x$ and $y$ belong to two different components, or $\gamma$ must pass through a maximum or a minimum. Since the length of $\gamma$ is less than or equal to the diameter, we can conclude as in the previous lemma that $\gamma(t)=\pm\sinp(t)$ on $[0,d]$. This in particular implies that $\ps{\dot \gamma}{\nabla u}=\pm\abs{\nabla u}$ at $t=0$, and since only two tangent vectors have this property, there can be only two points $y=\exp (x,\dot \gamma,d)$. Therefore the connected components of $N$ are discrete, and the manifold $M$ is one dimensional.

\paragraph{\textsc{Case 2, $N$ has only one component}} Now suppose that $N$ has only one connected component. Set $I=[-d/2,d/2]$ and define the function $h:N\times I \to M$ by
\begin{gather*}
 h(y,s)= \exp(y,X,s).
\end{gather*}
We will show that this function is a diffeomorphism and metric isometry.

First of all, let $\tilde h$ be the restriction of the map $h$ to $N\times I^\circ$ and note that if $\abs s < d/2$, by Lemma \ref{lemma_=}, $\tilde h(y,s)$ is the flux of the vector field $X$ emanating from $y$ evaluated at time $s$. Now it is easy to see that $u(h(y,s))=w(s)$ for all $s$. This is certainly true for all $y$ if $s=0$. For the other cases, fix $y$, let $\gamma(s)\equiv \tilde h(y,s)$ and $f(s)\equiv u(\tilde h(y,s))$. Then $f$ satisfies
\begin{gather*}
\dot \gamma = \frac{\nabla u}{\abs{\nabla u}} \ , \\
 \dot f = \ps{\nabla u}{\dot \gamma}= \abs {\nabla u}|_{\tilde h(y,s)} = \dot w |_{w^{-1} f(s)} \ .
\end{gather*}
Since $\dot w |_{w^{-1}(x)}$ is a smooth function, the solution of the above differential equation is unique and so $u(\tilde h(y,s))=f(s)=w(s)$ for all $y$ and $\abs s < d/2$. Note that by the continuity of $u$, we can also conclude that $u(h(y,s))=w(s)$ on all of $N\times I$.

The function $\tilde h$ is injective, in fact if $\tilde h(y,s)=\tilde h(z,t)$, then $w(s)=u(\tilde h(y,s))=u(\tilde h(z,t))=w(t)$ implies $s=t$. Moreover, since the flux of a vector at a fixed time in injective, also $y=z$.

Now we prove that $\tilde h$ is also a Riemannian isometry on its image. Let $\ps{\cdot}{\cdot}$ be the metric on $M$, $\ps{\cdot}{\cdot}_N$ the induced metric on $N$ and $\pps{\cdot}{\cdot}$ the product metric on $N\times I$. We want to show that $\pps{\cdot}{\cdot}=\tilde h^\star \ps{\cdot}{\cdot}$. The proof is similar in spirit to \cite[Lemma 9.7, step 7]{7i}. It is easily seen that
\begin{gather*}
 \tilde h^\star \ps{\cdot}{\cdot} (\partial s,\partial s)= \ps{\frac{\nabla u}{\abs{\nabla u}}}{\frac{\nabla u}{\abs{\nabla u}}}=1 \ ,
\end{gather*}
and for every $V\in T_{(y,s)} N$ we have
\begin{gather}\label{eq_perp}
\tilde  h^\star \psp (\partial s,V)=\ps{\frac{\nabla u}{\abs{\nabla u}}}{d\tilde h (V)}=0 \ .
\end{gather}
In fact, since $\abs{\nabla u}$ is constant on all level sets of $u$, if $\sigma(t)$ is a curve with image in $N\times \{s\}$ and with $\sigma(0)=(y,s)$ and $\dot \sigma (0)=V$, then
\begin{gather*}
 \frac{d}{dt} \ton{\abs{\nabla u} \circ \tilde h} \sigma(t)=0 \ .
\end{gather*}
Recalling that $H_u=-\frac{\lambda}{p-1} u^{(p-1)} \abs{\nabla u}^{p-2} \frac{\nabla u}{\abs{\nabla u}}\otimes \frac{\nabla u}{\abs{\nabla u}}$, we also have
\begin{gather*}
 \frac{d}{dt} \abs{\nabla u} \circ \ton{\tilde h( \sigma(t))}= -\frac{\lambda}{p-1} u^{(p-1)} \abs{\nabla u}^{p-2} \ps{\frac{\nabla u}{\abs{\nabla u}}}{d\tilde h(V)} \ ,
\end{gather*}
and the claim follows.

Fix any $V,W \in TN$, for any $\abs t <d/2$ note that, by the properties of the Lie derivative, we have
\begin{gather*}
  \left.\frac d {ds}\right\vert_{s=t} (\tilde h^\star \psp)  (V,W)=\L _{\partial s} [d\tilde h \psp] (V,W)=[d\tilde h \L_{X}\psp](V,W)=\\
= \ps{\nabla_{d\tilde h(V)} \frac{\nabla u}{\abs{\nabla u}} }{d\tilde h(W)} + \ps{dh(V)}{\nabla_{d\tilde h(W)} \frac{\nabla u}{\abs{\nabla u}}} \ .
\end{gather*}
It is easy to see that $d\tilde h(W) (\abs{\nabla u})=d\tilde h(V) (\abs{\nabla u})=0$ since $\abs{\nabla u}$ is constant on the level sets, and therefore
\begin{gather*}
\left.\frac d {ds}\right\vert_{s=t} (h^\star \psp)  (V,W) = 2 H_u (d\tilde h(V),d\tilde h(W))=0 \ .
\end{gather*}
This implies that for every $y\in N$ fixed and any $V,W\in TN$, $\tilde h^\star|_{y,s} \psp (V,W)$ is constant on $(-d/2,d/2)$, and since $\tilde h$ is a Riemannian isometry by definition on the set $N\times \{0\}$, we have proved that $h^\star\psp = \ppsp$.

Now, $h$ is certainly a differentiable map being defined as an exponential map, and it is the unique differentiable extension of $\tilde h$.


Injectivity and surjectivity for $h$ are a little tricky to prove, in fact consider the length space $N\times I/\sim$, where $(y,s)\sim(z,t)$ if and only if $s=t=\pm d/2$, endowed with the length metric induced by $\tilde h$. It is still possible to define $h$ as the continuous extension of $\tilde h$, and $N\times I/\sim$ is a length space of diameter $d$, but evidently $h$ is not injective. This shows that injectivity of $h$ has to be linked to some Riemannian property of the manifold $M$.

\paragraph{\textbf{$h$ is surjective}}For any point $x\in M$ such that $u(x)\neq \pm1$ the flux of the vector field $X$ joins $x$ with a point on the surface $N$ and vice versa, so $h$ is surjective on the set $E$. The set of points $u^{-1}(1)$ (and in a similar way the set $u^{-1}(-1)$) has empty interior since $u$ is an eigenfunction with positive eigenvalue. Fix any $x\in u^{-1}(1)$. The estimate $\abs{\nabla u}^p+\frac{\lambda}{p-1} \abs u ^p\leq \frac{\lambda}{p-1}$ implies that any geodesic ball $B_\epsilon(x)$ contains a point $x_\epsilon$ with $w(\pi_p/2-\epsilon)<u(x_\epsilon)<1$. Let $\gamma(t)$ be the flux of $X$ emanating from $x_\epsilon$. By the property of $X$, the curve $\gamma$ intersects $N$ in a single point $y_\epsilon$, moreover there exists a unique $z_\epsilon\in u^{-1}(1)$ which is an accumulation point for $\gamma$. If we define
\begin{gather*}
 \gamma(t)= \exp(y_\epsilon,X,t)=h(y_\epsilon,t)\\
f(t)= u(\gamma(t)) \ , 
\end{gather*}
we know that $\gamma$ is a minimizing geodesic on $[-d/2,d/2]$ and that $f(t)=w(t)=\sin_p(\alpha t)$. Since $u(x_\epsilon)>w(\pi_p/2-\epsilon)$
\begin{gather*}
 d(x_\epsilon,z_\epsilon)<w^{-1}(1)-   w^{-1}(w(\pi_p/2-\epsilon))=\epsilon \ .
\end{gather*}
Let $\epsilon$ go to zero and take a convergent subsequence of $\{y_\epsilon\}$ with limit $y$, then by continuity of the exponential map $h(y,d/2)=x$. Since $x$ was arbitrary, surjectivity is proved.

\paragraph{\textbf{$h$ is injective}}
Now we turn to the injectivity of $h$. Since $h$ is differentiable and its differential has determinant $1$ in $N\times I^\circ$, its determinant is $1$ everywhere and $h$ is a local diffeomorphism. By a similar density argument, it is also a local Riemannian isometry.
%
%
%
%
%
%
%
By the product structure on $N\times I$, we know that the parallel transport along a piecewise smooth curve $\sigma$ of the vector $X\equiv dh(\partial_s)$ is independent of $\sigma$. In particular if $\sigma$ is a loop, the parallel transport of $X$ along $\sigma$ is $\tau_\sigma(X)=X$.

Now consider two points $y,z\in N$ without any restriction on their mutual distance such that $h(y,d/2)=h(z,d/2)=x$. Let $\sigma$ be the curve obtained by gluing the geodesic $h(y,d/2-t)$ with any curve joining $x$ and $y$ in $M$ and with the geodesic $h(z,t)$. $\sigma$ is a loop around $x$ with
\begin{gather*}
 dh|_{(y,d/2)} \partial _s=X=\tau_\sigma(X)=dh|_{(z,d/2)} \partial _s \ .
\end{gather*}
Since by definition of $h$, $y=\exp(x,X,-d/2)$ and $z=\exp(x,\tau_\sigma X, -d/2)$, the equality $X=\tau_\sigma(X)$ implies $y=z$, and this proves the injectivity of $h$.

Now it is easily seen that $h$ is a metric isometry between $N\times I$ and $M$, which means that the diameter of $M$ is $d=\sqrt{d^2+diam(N)^2}$. Note that $diam(N)=0$ implies that $M$ is one dimensional (as in the case when $N$ has more than one connected component), and it is well-known that the only 1-dimensional connected compact manifolds are circles and segments.

As seen in Section \ref{sec_1d}, both these kind of manifolds realize equality in the sharp estimate for any diameter $d$, and so we have obtained our characterization.
\end{proof}

\subsection*{Acknowledgments}I especially thank Aaron (Moore Instructor @ MIT (USA)) since it's under his supervision and advice that this work started, and my advisor prof. Alberto Giulio Setti, Università dell'Insubria (EU),  for the useful insights he gave me. I would also like to thank prof. Stefano Pigola, Michele Rimoldi and Giona Veronelli for the interesting conversation we had on the problem.

After the first version of this article appeared on the web and was sent for publication, I also had some very interesting exchange with prof. Luca Esposito, prof. Carlo Nitsch and prof. Cristina Trombetti from the University Federico II of Naples (EU). I thank them for the suggestions they made and for their kindness. They published recently a very interesting article on sharp Poincarè constant estimates, which can be found at \cite{carlo}.

I also thank Dr. Songting Yin for pointing out the regularity issue explained in Remark \ref{rem_reg}.

\bibliographystyle{aomalpha}
\bibliography{problem_arxiv}
\end{document}